\documentclass{article}
\usepackage{amsthm,amsmath,amssymb,amsopn}
\usepackage{graphicx,epstopdf,epsfig,subfig}
\usepackage{algorithm,algorithmic}
\usepackage{verbatim}

\numberwithin{equation}{section}
\numberwithin{table}{section}

\newtheorem{proposition}{Proposition}[section]
\newtheorem{lemma}{Lemma}[section]
\newtheorem{theorem}{Theorem}[section]

\newtheorem{remark}{Remark}[section]
\newtheorem{assumption}{Assumption}[section]
\newtheorem{exam}{Example}[section]
\newtheorem{corollary}{Corollary}[section]
\theoremstyle{definition}

\begin{document}
\title{A Primal Dual Active Set with Continuation Algorithm for the $\ell^0$-Regularized Optimization Problem}

\author{
Yuling Jiao%
\thanks{School of Mathematics and Statistics, Wuhan University, Wuhan 430072, P.R. China.
(yulingjiaomath@whu.edu.cn)}
\and Bangti Jin%
\thanks{Department of Mathematics, University of California, Riverside, 900 University Ave.
Riverside, CA 92521 USA (bangti.jin@gmail.com)}
\and Xiliang Lu%
\thanks{Corresponding author. School of Mathematics and Statistics, Wuhan University,
Wuhan 430072, P.R. China (xllv.math@whu.edu.cn)}}

\maketitle

\begin{abstract}
We develop a primal dual active set with continuation algorithm for solving the $\ell^0$-regularized
least-squares problem that frequently arises in compressed sensing. The algorithm couples the the primal dual
active set method with a continuation strategy on the regularization parameter. At each
inner iteration, it first identifies the active set from both primal and dual variables, and then
updates the primal variable by solving a (typically small) least-squares problem defined on the active set,
from which the dual variable can be updated explicitly. Under certain conditions on the sensing
matrix, i.e., mutual incoherence property or restricted isometry property, and the noise level, the
finite step global convergence of the algorithm is established. Extensive numerical examples are presented
to illustrate the efficiency and accuracy of the algorithm and the convergence analysis. \\
\noindent\textbf{keywords:} primal dual active set method, coordinatewise minimizer,
continuation strategy, global convergence.
\end{abstract}

\section{Introduction}

Over the last ten years, compressed sensing  \cite{CandesRombergTao:2006,Donoho:2006} has received
a lot of attention amongst engineers, statisticians and mathematicians due to its broad
range of potential applications. Mathematically it can be formulated as the following $\ell^0$
optimization problem:
\begin{equation}\label{l0noise}
 \begin{aligned}
   & \min_{x \in \mathbb{R}^{p}}\|x\|_{0},\\
   \textrm{subject to}\ \ & \ \ \|\Psi x  -y\|_{2}\leq\epsilon,
 \end{aligned}
\end{equation}
where the sensing matrix $\Psi\in\mathbb{R}^{n\times p}$ with $p\gg n$ has normalized column vectors
(i.e., $\|\psi_i\| =1$, $i=1,\cdots,p$), $\epsilon\geq0$ is the noise level,
and $\|x\|_{0}$ denotes the the number of nonzero components
in the vector $x$. Due to the discrete structure of the term $\|x\|_0$, it is very challenging to
develop an efficient algorithm to accurately solve the model \eqref{l0noise}. 
Hence, approximate methods for the model \eqref{l0noise}, especially greedy heuristics
and convex relaxation, are very popular in practice. In greedy algorithms, including orthogonal matching pursuit
\cite{PatiRezaiifarKrishnaprasad:1993}, stagewise orthogonal matching pursuit \cite{DonohoTsaigDroriStarck:2012},
regularized orthogonal matching pursuit \cite{NeedellVershynin:2009}, CoSaMP
\cite{NeedellTropp:2009}, subspace pursuit \cite{DaiMilenkovic:2009}, and greedy
gradient pursuit \cite{BlumensathDavies:2009a} etc., one first identifies the support of the sought-for
signal, i.e., the locations of (one or more) nonzero components, iteratively based on the current
dual variable (correlation), and then updates the components on the support by solving a
least-squares problem. There are also several variants of the greedy heuristics, e.g., (accelerated) iterative hard thresholding
\cite{blumensath2012accelerated,BlumensathDavies:2009b} and hard thresholding pursuit \cite{Foucart:2011}, which are based
on the sum of the current primal and dual variable. In contrast, basis pursuit finds
one minimizer of a convex relaxation problem \cite{ChenDonohoSaunders:1998,Tibshirani:1996},
for which a wide variety of convex optimization algorithms can be conveniently applied;
see \cite{BJMG,pfbsr,proxrev,csnmr} for a comprehensive overview and the references therein.

Besides greedy methods and convex relaxation, the ``Lagrange'' counterpart of \eqref{l0noise} (or
equivalently, the $\ell^0$-regularized minimization problem), which reads
\begin{equation}\label{equ:l0reg}
  \min_{x \in \mathbb{R}^{p}} J_\lambda(x) = \tfrac{1}{2}\|\Psi x - y\|^2 + \lambda\|x\|_0,
\end{equation}
has been very popular in many applications, e.g., model selection, statistical regression, and
image restoration. In the model \eqref{equ:l0reg}, $\lambda>0$ is a regularization parameter,
controlling the sparsity level of the regularized solution. Due to the nonconvexity and
discontinuity of the function $\|x\|_0$, the relation between problems \eqref{l0noise} and
\eqref{equ:l0reg} is not self evident. We shall show that under certain assumptions on the sensing
matrix $\Psi$ and the noise level $\epsilon$ (and with $\lambda$ chosen properly), the support
of the solution to \eqref{equ:l0reg} coincides with that of the true signal, cf. Theorem \ref{thm:global}.

Very recently, the existence and a characterization of global minimizers to \eqref{equ:l0reg} were
established in \cite{JiaoJinLu:2013,Nikolova:2013}. However, it is still very challenging to develop
globally convergent algorithms for efficiently solving problem \eqref{equ:l0reg} in view of its
nonconvexity and nonsmoothness. Nonetheless, due to its broad range of applications, several algorithms
have been developed to find an approximate solution to problem \eqref{equ:l0reg},
including iterative hard thresholding \cite{BlumensathDavies:2009b}, forward backward splitting
\cite{AttouchBolteSvaiter:2013}, penalty decomposition \cite{LuZhang:2013} and stochastic continuation
\cite{RobiniMagnin:2010,RobiniReissman:2013}, to name just a few. 
Theoretically, these algorithms can at best have a local convergence. Very recently, in
\cite{ItoKunisch:2014,JiaoJinLu:2013}, based on a coordinatewise characterization of the global
minimizers, a novel primal dual active set (PDAS) algorithm was developed to solve problem \eqref{equ:l0reg}.
The extensive simulation studies in \cite{JiaoJinLu:2013} indicate that when coupled with a continuation
technique, the PDAS algorithm merits a global convergence property.

We note that the PDAS can at best converge to a coordinatewise minimizer. However, if the support
of the coordinatewise minimizer is small and the sensing matrix $\Psi$ satisfies certain mild conditions,
then its active set is contained in the support of the true signal, cf. Lemma \ref{lem:cwm}. Hence, the support
of the minimizer will coincide with that of the true signal if we choose the parameter $\lambda$ properly (and
thus control the size of the active set) during the iteration. This motivates the use of a continuation strategy
on the parameter $\lambda$. The resulting PDAS continuation (PDASC) algorithm
extends the PDAS developed in \cite{JiaoJinLu:2013}. In this work, we provide a convergence analysis of
the PDASC algorithm 
under commonly used assumptions on the sensing matrix $\Psi$ for the analysis of existing algorithms, i.e., mutual incoherence property and
restricted isometry property. The convergence analysis relies essentially on
a novel characterization of the evolution of the active set during the primal-dual active set iterations.
To the best of our knowledge, this represents the first work on the global convergence of an algorithm for
problem \eqref{equ:l0reg}, without using a knowledge of the exact sparsity level.

The rest of the paper is organized as follows. In Section \ref{sec:prelim}, we describe the
problem setting, collect basic estimates, and provide
a refined characterization of a coordinatewise minimizer. In Section \ref{sec:alg}, we
give the complete algorithm, discuss the parameter choices, and provide a convergence analysis.
Finally, in Section \ref{sec:numeric}, several numerical examples are provided to illustrate
the efficiency of the algorithm and the convergence theory.

\section{Regularized $\ell^0$-minimization}\label{sec:prelim}
In this section, we describe the problem setting, and derive basic estimates, which are essential
for the convergence analysis. Further, we give sufficient conditions for a coordinatewise minimizer
to be a global minimizer.

\subsection{Problem setting}
Suppose that the true signal $x^*$ has $T$ nonzero components with its
active set (indices of nonzero components) denoted by $A^*$, i.e., $T=|A^*|$ and
the noisy data $y$ is formed by
\begin{equation*}
  y = \sum_{i\in{A}^*} x_i^*\psi_i + \eta.
\end{equation*}
We assume that the noise vector $\eta$ satisfies $\|\eta\|\leq \epsilon$, with $\epsilon\geq0$ being
the noise level. Further, we let
\begin{equation*}
  S = \{1,2,...,p\}\quad\mbox{and}\quad I^* = S\backslash{A}^*.
\end{equation*}
For any index set ${A}\subseteq S$, we denote by $x_{A}\in\mathbb{R}^{|{A}|}$ (respectively
$\Psi_{{A}} \in\mathbb{R}^{n\times |{A}|}$) the subvector of $x$ (respectively the submatrix of
$\Psi$) whose indices (respectively column indices) appear in ${A}$. Last, we denote by $x^o$ the
oracle solution defined by
\begin{equation}\label{equ:oracle}
x^o = \Psi_{{A}^*}^\dag y,
\end{equation}
where $\Psi_A^\dag$ denotes the pseudoinverse of the submatrix $\Psi_A$, i.e., $\Psi_A^\dag=(\Psi^t_{A}\Psi_{A})^{-1}
\Psi^t_{{A}}$ if $\Psi^t_{A}\Psi_{A}$ is invertible.

In compressive sensing, there are two assumptions, i.e., mutual incoherence property (MIP)
\cite{DonohoHuo:2001} and restricted isometry property (RIP) \cite{CandesTao:2005}, on the
sensing matrix $\Psi$ that are frequently used for the convergence analysis of sparse recovery
algorithms. The MIP relies on the fact that the mutual coherence (MC) $\nu$ of sensing matrix $\Psi$ is
small, where the mutual coherence (MC) $\nu$ of $\Psi$ is defined by
\begin{equation*}
\nu = \max\limits_{1\leq i, j \leq p, i\neq j} |\psi_i^t\psi_j|.
\end{equation*}
A sensing matrix $\Psi$ is said to satisfy RIP of level $s$ if there exists a constant $\delta\in (0,1)$ such that
\begin{equation*}
(1-\delta)\|x\|^2 \leq \|\Psi x\|^2 \leq (1+\delta)\|x\|^2, \,\, \forall x\in\mathbb{R}^p \mbox{ with } \|x\|_0 \leq s,
\end{equation*}
and we denote by $\delta_s$ the smallest constant with respect to the sparsity level $s$. We note that
the mutual coherence $\nu$ can be easily computed, but the RIP constant $\delta_s$ is
nontrivial to evaluate.

The next lemma gives basic estimates under the MIP condition.
\begin{lemma}\label{lem:mc}
Let $A$ and $B$ be disjoint subsets of $S$. Then
\begin{equation*}
  \begin{aligned}
    \|\Psi^t_A y\|_{\ell^\infty}&\leq \|y\|,\\
    \|\Psi^t_B\Psi_A x_A\|_{\ell^\infty}& \leq |A|\nu\|x_A\|_{\ell^\infty},\\
   \|(\Psi_A^t\Psi_A)^{-1} x_A\|_{\ell^\infty} &\leq \frac{\|x_A\|_{\ell^\infty}}{1 - (|A|-1)\nu} \quad \mbox{if} \quad
   (|A| -1)\nu <1.
  \end{aligned}
\end{equation*}
\end{lemma}
\begin{proof}
If $A = \emptyset$, then the estimates are trivial. Hence we will assume $A$ is nonempty. For any $i\in A$,
\begin{equation*}
  |\psi_i^t y| \leq \|\psi_i\|\|y\|\leq \|y\|.
\end{equation*}
This shows the first inequality. Next, for any $i\in B$,
\begin{equation*}
  |\psi_i^t\Psi_A x_A| = |\sum_{j\in A} \psi_i^t\psi_jx_j|\leq \sum_{j\in A}|\psi_i^t\psi_j||x_j|\leq |A|\nu \|x_A\|_{\ell^\infty}.
\end{equation*}
This shows the second assertion. To prove the last estimate, we follow the proof strategy
of \cite[Theorem 3.5]{Tropp:2004}, i.e., applying a Neumann series method. First we note
that $\Psi_A^t\Psi_A$ has a unit diagonal because all columns of $\Psi$ are normalized. So
the off-diagonal part $\Phi$ satisfies
\begin{equation*}
  \Psi_A^*\Psi_A = E_{|A|} + \Phi,
\end{equation*}
where $E_{|A|}$ is an identity matrix.
Each column of the matrix $\Phi$ lists the inner products between one column
of $\Psi_A$ and the remaining $|A|-1$ columns. By the definition
of the mutual coherence $\nu$ and the operator norm of a matrix
\begin{equation*}
  \|\Phi\|_{\ell^\infty,\ell^\infty} = \max_{k\in A}\sum_{j\in A\setminus\{k\}} |\psi_j^t\psi_k|\leq (|A|-1)\nu.
\end{equation*}
Whenever $\|\Phi\|_{\ell^\infty,\ell^\infty}<1$, the Neumann series $\sum_{k=0}^\infty(-\Phi)^k$ converges to the inverse $(E_{|A|}+\Phi)^{-1}$.
Hence, we may compute
\begin{equation*}
  \begin{aligned}
    \|(\Psi_A^*\Psi_A)^{-1}\|_{\ell^\infty,\ell^\infty} & = \|(E_{|A|}+\Phi)^{-1}\|_{\ell^\infty,\ell^\infty}
    = \|\sum_{k=0}^\infty (-\Phi)^k\|_{\ell^\infty,\ell^\infty}\\
   &\leq \sum_{k=0}^\infty \|\Phi\|_{\ell^\infty,\ell^\infty}^k=\frac{1}{1-\|\Phi\|_{\ell^\infty,\ell^\infty}} \leq \frac{1}{1-(|A|-1)\nu}.
 \end{aligned}
\end{equation*}
The desired estimate now follows immediately.
\end{proof}

The following lemma collects some estimates on the RIP constant $\delta_s$; see \cite[Propositions 3.1
and 3.2]{NeedellTropp:2009} and  \cite[Lemma 1]{DaiMilenkovic:2009} for the proofs.
\begin{lemma}\label{lem:rip}
Let $A$ and $B$ be disjoint subsets of $S$. Then
\begin{equation*}
  \begin{aligned}
    \|\Psi_A^t\Psi_A x_A\| &\gtreqqless (1\mp \delta_{|A|})\|x_A\|,\\
    \|(\Psi_A^t\Psi_A)^{-1} x_A\| &\gtreqqless \frac{1}{1\pm \delta_{|A|}}\|x_A\|,\\
    \|\Psi_A^t\Psi_B\| &\leq \delta_{|A|+|B|},\\
    \|\Psi_A^\dag y\|& \leq \frac{1}{\sqrt{1-\delta_{|A|}}}\|y\|,\\
     \delta_{s} & \leq \delta_{s'}, \mbox{ if } s<s'. 
  \end{aligned}
\end{equation*}
\end{lemma}

The next lemma gives some crucial estimates for one-step primal dual active set iteration on the active set $A$.
These estimates provide upper bounds on the dual variable $d=\Psi^t(y-\Psi x)$ and the error $\bar{x}_A=x_A-x^*_A$
on the active set $A$. They will play an essential role for subsequent analysis, including the convergence of the algorithm.
\begin{lemma}\label{lem:est-onestep}
For any set ${A}\subseteq S$ with $|{A}|\leq T$, let 
${B} = {A}^*\setminus {A}$ and $I=S\setminus {A}$, and consider the following primal dual iteration on $A$
\begin{equation*}
x_{{A}} = \Psi_{{A}}^\dag y, \quad x_{{I}} = 0,\quad d = \Psi^t( y -\Psi x).
\end{equation*}
Then the quantities $\bar{x}_A\equiv x_{{A}} - x^*_{{A}}$ and $d$ satisfy the following estimates.
\begin{itemize}
\item[(a)] If $\nu < 1/(T-1)$, then $d_A = 0 $ and
\begin{equation*}
  \begin{aligned}
       \|\bar{x}_A\|_{\ell^\infty}& \leq \frac{1}{1-(|A| -1)\nu}\left(|B|\nu \|x_B^*\|_{\ell^\infty} + \epsilon\right),\\ 
     |d_{j}| & \geq |x_j^*| - \|x_B^*\|_{\ell^\infty}(|B|-1)\nu - \epsilon - |A|\nu \|\bar{x}_A\|_{\ell^\infty}, \quad \forall j\in B,\\ 
     |d_j| &\leq |B|\nu \|x_B^*\|_{\ell^\infty} + \epsilon + |A|\nu\|\bar{x}_A\|_{\ell^\infty}, \quad \forall j\in I^*\cap I. 
  \end{aligned}
\end{equation*}
\item[(b)] If the RIP is satisfied for sparsity level $s:=\max\{|A| + |B|, T+1\}$, then $d_A = 0$ and
\begin{equation*}
   \begin{aligned}
     \|\bar{x}_{A}\|& \leq \frac{\delta_{|{A}| + |{B}|}}{1-\delta_{|{A}|}}\|x^*_{B}\| + \frac{1}{\sqrt{1-\delta_{|{A}|}}}\epsilon,\\ 
     |d_{j}| & \geq |x^*_j| - \delta_{|B|}\|x_{{B}}^*\| - \epsilon - \delta_{|A|+1}\|\bar{x}_{{A}}\| , \quad \forall j\in B,\\
     |d_j|&\leq \delta_{|B|+1}\|x_{{B}}^*\| + \epsilon + \delta_{|A|+1}\|\bar{x}_{{A}}\| , \quad \forall j\in{I}^*\cap{I},
   \end{aligned}
\end{equation*}
\end{itemize}
\end{lemma}
\begin{proof}
We show only the estimates under the RIP condition and using Lemma \ref{lem:rip}, and that for the MIP
condition follows similarly from Lemma \ref{lem:mc}. If $A= \emptyset$, then all the estimates clearly hold.
In the case $A\neq \emptyset$, then by the assumption, $\Psi_A^t\Psi_A$ is invertible. By the definition of the update $x_A$ and
the data $y$ we deduce that
\begin{equation*}
d_A = \Psi_A^t( y - \Psi_A x_A) = 0,
\end{equation*}
and
\begin{equation*}
  \begin{aligned}
    \bar{x}_{{A}} &= (\Psi^t_{A}\Psi_{A})^{-1}\Psi^t_{{A}} (\Psi_{{A}^*} x_{{A}^*}^* + \eta - \Psi_{{A}} x_{{A}}^*)\\
      &= (\Psi^t_{A}\Psi_{A})^{-1}\Psi^t_A (\Psi_{B} x^*_{B} + \eta).
  \end{aligned}
\end{equation*}
Consequently, by Lemma \ref{lem:rip} and the triangle inequality, there holds
\begin{equation*}
  \begin{aligned}
    \|\bar{x}_A\| & \leq \frac{1}{1-\delta_{|A|}}\|\Psi^t_A\Psi_B x_B^*\| +\|\Psi_A^\dag \eta\|\\
      & \leq \frac{1}{1-\delta_{|A|}}\delta_{|A|+|B|}\|x_B^*\| + \frac{1}{\sqrt{1-\delta_{|A|}}}\epsilon.
  \end{aligned}
\end{equation*}
It follows from the definition of the dual variable $d$, i.e.,
\begin{equation*}
   d_j = \psi^t_j(y - \Psi_{A}x_{A}) = \psi^t_j(\Psi_{B} x^*_{{B}} + \eta - \Psi_{A} \bar{x}_{A}),
\end{equation*}
Lemma \ref{lem:rip}, and $\psi_j^t\psi_j=1$ that for any $j\in B$, there holds
\begin{equation*}
  \begin{aligned}
    |d_j| & = |\psi^t_j\psi_j x_j^* + \psi_j^t(\Psi_{B\setminus\{j\}} x_{B\setminus\{j\}}^* + \eta - \Psi_A\bar{x}_A)|\\
       & \geq |x_j^*| - (|\psi_j^t\Psi_{B\setminus\{j\}} x_{B\setminus\{j\}}^*| + |\psi_j^t\eta| + |\psi_j^t\Psi_A\bar{x}_A|)\\
       & \geq |x_j^*| - \delta_{|B|}\|x_B^*\| - \epsilon - \delta_{|A|+1}\|\bar{x}_{A}\|.
  \end{aligned}
\end{equation*}
Similarly, for any $j\in I^*\cap I $, there holds
\begin{equation*}
  |d_j| \leq \delta_{|B|+1} \|x^*_B\|+\epsilon + \delta_{|A|+1}\|\bar{x}_A\|.
\end{equation*}
This completes the proof of the lemma.
\end{proof}

\subsection{Coordinatewise minimizer}
Next we characterize minimizers to problem \eqref{equ:l0reg}. Due to the nonconvexity and discontinuity of
the function $\|x\|_0$, the classical theory \cite{ItoKunisch:2008} on the existence of a Lagrange multiplier
cannot be applied directly to show the equivalence between problem \eqref{equ:l0reg} and the Lagrange counterpart \eqref{l0noise}.
Nonetheless, both formulations aim at recovering the true sparse signal $x^*$, and thus we expect that
they are closely related to each other. We shall establish below that with the parameter $\lambda$
properly chosen, the oracle solution $x^o$ is the only global minimizer of problem \eqref{equ:l0reg},
and as a consequence, we derive directly the equivalence between problems \eqref{l0noise} and \eqref{equ:l0reg}.

To this end, we first characterize the minimizers of problem \eqref{equ:l0reg}. Since the cost
function $J_\lambda(x)$ is nonconvex and discontinuous, instead of a global minimizer, we study its
coordinatewise minimizers, following \cite{Tseng:2001}. A vector $x=(x_1,x_2, \dots,x_p)^t\in\mathbb{R}^p$
is called a coordinatewise minimizer to $J_\lambda(x)$ if it is the minimum along
each coordinate direction, i.e.,
\begin{equation*}
x_i \in \mathop\mathrm{arg}\min\limits_{t\in\mathbb{R}}  J_\lambda(x_1,...,x_{i-1},t,x_{i+1},...,x_p).
\end{equation*}
The necessary and sufficient condition for a coordinatewise minimizer $x$ is given by \cite{ItoKunisch:2014,JiaoJinLu:2013}:
\begin{equation}\label{equ:con1}
x_i \in S^{\ell^0}_\lambda(x_i + d_i)\quad \forall i\in S,
\end{equation}
where $d = \Psi^t(y - \Psi x)$ denotes the dual variable, and $S^{\ell^0}_\lambda$ is the hard thresholding operator defined by
\begin{equation}\label{equ:thres}
S^{\ell^0}_{\lambda}(v)  \left\{\begin{array}{ll} =0, & |v| < \sqrt{2\lambda}, \\[1.3ex]
 \in \{0, \mbox{sgn}(v)\sqrt{2\lambda}\}, & |v| = \sqrt{2\lambda}, \\[1.3ex]
 = v, & |v| > \sqrt{2\lambda}. \end{array}\right.
\end{equation}
The condition \eqref{equ:con1} can be equivalently written as
\begin{equation*}
\left\{\begin{array}{l}
|x_i + d_i| > \sqrt{2\lambda} \Rightarrow d_i = 0,\\[1.3ex]
|x_i + d_i| < \sqrt{2\lambda} \Rightarrow x_i = 0,\\[1.3ex]
|x_i + d_i| = \sqrt{2\lambda} \Rightarrow x_i = 0 \mbox{ or } d_i =0.
\end{array}
\right.
\end{equation*}
Consequently, with the active set ${A} = \{i: x_i\neq 0\}$, there holds
\begin{equation}\label{equ:con2}
\min\limits_{i\in{A}} |x_i| \geq \sqrt{2\lambda} \geq \|d\|_{\ell^\infty}.
\end{equation}

It is known that any coordinatewise minimizer $x$ is a local minimizer \cite{JiaoJinLu:2013}.
To further analyze the coordinatewise minimizer, we need the following assumption on the noise level $\epsilon$:
\begin{assumption}\label{ass:epsilon}
The noise level $\epsilon$ is small in the sense $\epsilon\leq \beta\min_{i\in A^*}|x_i^*|$, $0\leq \beta<1/2$.
\end{assumption}

The next lemma gives an interesting characterization of the active set of the coordinatewise minimizer.
\begin{lemma}\label{lem:cwm}
Let Assumption \ref{ass:epsilon} hold, and $x$ be a coordinatewise minimizer with support ${A}$ and
$|A|\leq T$. If either (a) $\nu < (1-2\beta)/(3T - 1)$ or (b) $\delta \triangleq \delta_{2T} \leq
(1-2\beta)/(2\sqrt{T}+1)$ holds, then ${A}\subseteq {A}^*$.
\end{lemma}
\begin{proof}
Let $I=S\setminus A$. Since $x$ is a coordinatewise minimizer, it follows from \eqref{equ:con2} that
\begin{equation*}
   x_A = \Psi_A^\dag y,\quad x_I=0, \quad d = \Psi^t(y - \Psi x).
\end{equation*}
We prove the assertions by means of contradiction. Assume the contrary, i.e., $A
\nsubseteq A^*$. We let $B=A^*\setminus A$, which is nonempty by assumption, and
denote by $i_{A} \in \{i\in I: {|x_i^*| }= \|x^*_B\|_{\ell^\infty}\}$. Then $i_A
\in B$. Further by \eqref{equ:con2}, there holds
\begin{equation}\label{eqn:dcomx}
   |d_{i_A}| \leq \|d\|_{\ell^\infty} \leq \min_{i\in A}|x_i| \leq \min_{i\in A\backslash A^*} |x_i| \leq \|\bar{x}_A\|_{\ell^\infty} \leq \|\bar{x}_A\|.
\end{equation}
Now we discuss the two cases separately.\\
\noindent {\bf Case (a)}. By Lemma \ref{lem:mc}, $\epsilon \leq \beta \min_{i\in A^*}|x_i^*| \leq \beta
\|x_B^*\|_{\ell^\infty}$ from Assumption \ref{ass:epsilon} and the choice of the index $i_A$, we have
\begin{equation*}
  \begin{aligned}
    \|\bar{x}_A\|_{\ell^\infty} &\leq  \frac{1}{1-(|A| -1)\nu}\left(|B|\nu \|x_B^*\|_{\ell^\infty} + \epsilon\right) \\
     & \leq \frac{1}{1-(|A|-1)\nu}(|B|\nu+\beta)\|x_B^*\|_{\ell^\infty},\\
    |d_{i_A}| & \geq  |x_{i_A}^*| - \|x_B^*\|_{\ell^\infty}(|B|-1)\nu - \epsilon - |A|\nu \|\bar{x}_A\|_{\ell^\infty}\\
         & \geq \|x_B^*\|_{\ell^\infty}\left(1-(|B|-1)\nu - \beta - |A|\nu \frac{1}{1-(|A|-1)\nu}(|B|\nu+\beta)\right).
  \end{aligned}
\end{equation*}
Consequently, we deduce
\begin{equation*}
  \begin{aligned}
    |d_{i_A}| - \|\bar{x}_A\|_{\ell^\infty} &\geq
    \frac{\|x^*_B\|_{\ell^\infty}}{1 - (|A|-1)\nu}\left[ 1- (|A| + 2|B|)\nu - (|A|+|B|)\nu^2 + 2\nu +\nu^2 - \beta(\nu+2)\right]\\
   & \geq \frac{\|x^*_B\|_{\ell^\infty}}{1-(|A|-1)\nu}\left[1-3T\nu + \nu-2\beta + \nu(1 - \beta - 2T\nu)\right]\\
    & \geq \frac{\|x^*_B\|_{\ell^\infty}}{1-(|A|-1)\nu}\left[1-(3T-1)\nu-2\beta\right]>0,
  \end{aligned}
\end{equation*}
under assumption (a) $\nu<(1-2\beta)/(3T-1)$. 
This leads to a contradiction to \eqref{eqn:dcomx}.\\
\noindent {\bf Case (b)}. By assumption, $|A| + |B|\leq 2T$ and by Lemma
\ref{lem:rip}, there hold
\begin{equation*}
  \begin{aligned}
     \|\bar{x}_A\| &\leq \frac{\delta}{1-\delta}\|x^*_B\| + \frac{1}{\sqrt{1-\delta}}\epsilon\\
       &\leq \frac{\delta}{1-\delta}\|x^*_B\| + \frac{1}{1-\delta}\epsilon,\\
    |d_{i_A}|& \geq |x^*_{i_A}| - \delta\|x_{{B}}^*\| - \epsilon - \delta\|\bar{x}_{{A}}\| \\
       & \geq |x^*_{i_A}| - \frac{\delta}{1-\delta}\|x_B^*\| - \frac{1}{1-\delta}\epsilon.
  \end{aligned}
\end{equation*}
Consequently, with the assumption on $\epsilon$ and $\delta<\frac{1-2\beta}{2\sqrt{T}+1}$, we get
\begin{equation*}
  \begin{aligned}
    |d_{i_A}| - \|\bar{x}_A\| &\geq |x^*_{i_A}| - \frac{2\delta}{1-\delta}\|x_B^*\| - \frac{2}{1-\delta}\epsilon \\
   &\geq  |x^*_{i_A}|\left(1 - \frac{2\sqrt{T}\delta + 2\beta}{1-\delta}\right) > 0,
  \end{aligned}
\end{equation*}
which is also a contradiction to \eqref{eqn:dcomx}. This completes the proof of the lemma.
\end{proof}

From Lemma \ref{lem:cwm}, it follows if the support size of the active set of the coordinatewise minimizer can be
controlled, then we may obtain information of the true active set $A^*$. However, a local minimizer generally
does not yield such information; see following result. The proof can be found also in \cite{Nikolova:2013},
but we include it here for completeness.

\begin{proposition}\label{prop:localmin}
for any given index set $A\subseteq S$, the solution $x$ to the least-squares problem
$\mathrm{min}_{\mathrm{supp}(x)\subseteq A} \|\Psi x - y\|$
is a local minimizer.
\end{proposition}
\begin{proof}
Let $\tau = \min \{|x_i|: x_i\neq 0\}$. Then for any small perturbation $h$ in the sense $\|h\|_{\ell^\infty}
<\tau$, we have $x_i \neq 0\rightarrow x_i + h_i \neq 0$. Now we show that $x$ is a local minimizer.
To see this, we consider two cases. First consider the case $\mathrm{supp}(h)\subseteq A$. By the definition
of $x$, and $\|x\|_0 \leq \|x+h\|_0$, we deduce
\begin{equation*}
  \begin{aligned}
    J_\lambda(x+h) &= \tfrac{1}{2}\|\Psi(x+h) - y\|^2 + \lambda\|x+h\|_0\\
     & \geq \tfrac{1}{2}\|\Psi x - y\|^2 + \lambda\|x\|_0  = J_\lambda(x).
  \end{aligned}
\end{equation*}
Alternatively, if $\mathrm{supp}(h)\nsubseteq A$, then $\|x+h\|_0 \geq \|x\|_0 + 1$. Since
\begin{equation*}
  \lim\limits_{\|h\|\rightarrow 0}\|\Psi (x+h) - y\| = \|\Psi x - y\|,
\end{equation*}
we again have $J_\lambda(x+h) > J_\lambda(x)$ for sufficiently small $h$. This completes the proof
of the proposition.
\end{proof}

Now we can study global minimizers to problem \eqref{equ:l0reg}. For any $\lambda > 0$, there
exists a global minimizer $x_\lambda$ to problem \eqref{equ:l0reg} \cite{JiaoJinLu:2013}. Further, the
following monotonicity relation holds \cite{ItoJinTakeuchi:2011}\cite[Section 2.3]{ItoJin:2014}.
\begin{lemma}\label{lem:mon}
For $\lambda_1>\lambda_2>0$, there holds
$\|x_{\lambda_1}\|_0 \leq \|x_{\lambda_2}\|_0.$
\end{lemma}

If the noise level $\epsilon$ is sufficiently small, and the parameter $\lambda$ is properly chosen,
the oracle solution $x^o$ is the only global minimizer to $J_\lambda(x)$, cf. Theorem \ref{thm:global},
which in particular implies the equivalence between the two formulations  \eqref{l0noise} and \eqref{equ:l0reg};
see Remark \ref{rem:eqv} below.

\begin{theorem}\label{thm:global}
Let Assumption \ref{ass:epsilon} hold.
\begin{itemize}
\item[(a)] Suppose $\nu < (1-2\beta)/(3T - 1)$ and $\beta \leq (1- 2(T-1)\nu)/({T+3})$, and let
\begin{equation*}
\xi = \frac{1 - 2(T-1)\nu - 2\beta - \beta^2}{2T}\min_{i\in A^*}|x_i^*|^2.
\end{equation*}
Then for any $\lambda\in (\epsilon^2/2, \xi)$, $x^o$ is the only global minimizer to $J_\lambda(x)$.
\item[(b)] Suppose $\delta \triangleq \delta_{2T} \leq (1-2\beta)/({2\sqrt{T}+1})$ and $\beta \leq
(1-2\delta - \delta^2)/4$, and let
\begin{equation*}
   \xi = \left[\frac{1}{2}(1-\delta) - \frac{\delta^2}{1-\delta} - \frac{\beta}{\sqrt{1-\delta}} -
   \frac{1}{2}\beta^2\right]\min_{i\in A^*}|x_i^*|^2.
\end{equation*}
Then for any $\lambda\in (\epsilon^2/2, \xi)$, $x^o$ is the only global minimizer to $J_\lambda(x)$.
\end{itemize}
\end{theorem}
\begin{proof}
Let $x$ be a global minimizer to problem \eqref{equ:l0reg}, and its support be $A$.
It suffices to show $A=A^*$.
If $|A| \geq T+1$, then by the choice of $\lambda$, we deduce
\begin{equation*}
J_\lambda(x) \geq \lambda(T+1) > \lambda T + \tfrac{1}{2}\epsilon^2 \geq J_\lambda (x^o),
\end{equation*}
which contradicts the minimizing property of $x$. Hence, $|A| \leq T$. Since a global minimizer
is always a coordinatewise minimizer, by Lemma \ref{lem:cwm}, we deduce $A\subseteq A^*$.
If $A\neq A^*$, then $B = A^*\backslash A$ is nonempty. By the global minimizing property
of $x$, there holds $x = \Psi_A^\dag y$.
Using the notation $\bar{x}_A$ from Lemma \ref{lem:est-onestep}, we have
\begin{equation}\label{eqn:Jfcn}
   J_\lambda(x) = \tfrac{1}{2}\|\Psi_B x^*_B + \eta - \Psi_A \bar{x}_A\|^2 + \lambda |A|.
\end{equation}
Now we consider the cases of the MIP and RIP separately.\\
Case (a): Let  $i_{A} \in \{i\in A^c: {|x_i^*| }= \|x^*_B\|_{\ell^\infty}\}$, then $i_A\in B$ and
$|x_{i_A}^*| = \|x^*_B\|_{\ell^\infty}$. Hence, by Lemma \ref{lem:est-onestep}, there holds
\begin{equation*}
  \begin{aligned}
      \tfrac{1}{2}\|\psi_{i_A}x^*_{i_A} &+ \Psi_{B\backslash \{i_A\}} x^*_{B\backslash \{i_A\}} + \eta - \Psi_A \bar{x}_A\|^2 \\
      \geq& \tfrac{1}{2}|x_{i_A}^*|^2 - |x_{i_A}^*|\left(|\langle \psi_{i_A},\Psi_{B\backslash \{i_A\}} x^*_{B\backslash \{i_A\}}\rangle| + |\langle \psi_{i_A}, \eta \rangle| + |\langle \psi_{i_A}, \Psi_A \bar{x}_A\rangle|\right)\\
      \geq& \tfrac{1}{2}|x_{i_A}^*|^2 - |x_{i_A}^*|\left((|B|-1)\nu |x_{i_A}^*| + \epsilon + \frac{|A|\nu }{1 - (|A| -1)\nu}(|B|\nu |x_{i_A}^*| + \epsilon) \right).
  \end{aligned}
\end{equation*}
Now with $\epsilon <\beta \min_{i\in A^*} |x_i^*|\leq \beta |x_{i_A}^*|$ from Assumption \eqref{ass:epsilon}, we deduce
\begin{equation*}
  \begin{aligned}
     J_\lambda(x) \geq& |x_{i_A}^*|^2\left(\frac{1}{2} - \left((|B|-1)\nu + \beta + \frac{|A|\nu }{1 - (|A| -1)\nu}(|B|\nu + \beta) \right)\right) + \lambda|A| \\
      =  & |x_{i_A}^*|^2\left(\frac{1}{2} - (T-1)\nu  -\beta\right) + |x_{i_A}^*|^2 |A|\nu \left(1 - \frac{|B|\nu + \beta}{1- (|A|-1)\nu}\right) +\lambda|A| \\
        \geq & |x_{i_A}^*|^2\left(\frac{1}{2} - (T-1)\nu  -\beta\right),
  \end{aligned}
\end{equation*}
where the last inequality follows from $(|A| + |B|-1)\nu + \beta < 1 $.
By Assumption \ref{ass:epsilon}, there holds $\epsilon^2/2 \leq \beta^2/2
\min_{i\in A^*}|x_i^*|^2$. Now by the assumption $\beta \leq \left(1- 2(T-1)\nu\right)/(T+3)$,
we deduce $(T+1)\beta^2 + 2\beta<(T+3)\beta \leq 1-2(T-1)\nu$, and hence $T\beta^2 <
1-2(T-1)\nu-2\beta-\beta^2$. Together with the definition of $\xi$, this implies
$\xi>\epsilon^2/2$. Further, by the choice of the parameter $\lambda$, i.e.,
$\lambda\in (\epsilon^2/2,\xi)$, there holds
\begin{equation*}
  J_\lambda(x) - J_\lambda(x^o) \geq \left[\frac{1}{2} - (T-1)\nu - \beta -
 \frac{1}{2}\beta^2\right]\min_{i\in A^*}|x_i^*|^2 - \lambda T >0,
\end{equation*}
which contradicts the optimality of $x$.

\noindent Case (b): It follows from \eqref{eqn:Jfcn} that
\begin{equation*}
  \begin{aligned}
  J_\lambda(x)& \geq \tfrac{1}{2}\|\Psi_B x_B^*\|^2 - |\langle \eta, \Psi_B x_B^*\rangle| - |\langle x_B^*,\Psi_B^t\Psi_A \bar{x}_A\rangle|  + \lambda |A|\\
    &\geq \|\Psi_B x_B^*\|(\tfrac{1}{2}\|\Psi_B x_B^*\| - \epsilon) - \|x_B^*\|\delta\|\bar{x}_A\| + \lambda |A|.
  \end{aligned}
\end{equation*}
By Assumption \ref{ass:epsilon} and the assumptions on $\beta$ and $\delta$, we deduce
$\sqrt{1-\delta}\|x_B^*\| \geq \epsilon$. Now in view of the monotonicity of the function
$t(t/2 - \epsilon)$ for $t\geq \epsilon$, and the inequality $\|\Psi_B x_B^*\|
\geq \sqrt{1-\delta}\|x_B^*\|$ from the definition of the RIP constant $\delta$, we have
\begin{equation*}
\|\Psi_B x_B^*\|(\tfrac{1}{2}\|\Psi_B x_B^*\| - \epsilon) \geq \sqrt{1-\delta}\|x_B^*\|(\tfrac{1}{2}\sqrt{1-\delta}\|x_B^*\| - \epsilon).
\end{equation*}
Thus by Lemma \ref{lem:est-onestep}, we deduce
\begin{equation*}
  \begin{aligned}
    J_\lambda(x)& \geq \frac{1-\delta}{2}\|x_B^*\|^2 - \epsilon\sqrt{1-\delta}\|x_B^*\| - \|x_B^*\|\left(\frac{\delta^2}{1-\delta}\|x_B^*\| + \frac{\delta}{\sqrt{1-\delta}}\epsilon\right) + \lambda |A|\\
    & = \frac{1-\delta}{2}\|x_B^*\|^2 - \frac{1}{\sqrt{1-\delta}}\epsilon\|x_B^*\| - \|x_B^*\|^2\frac{\delta^2}{1-\delta} + \lambda |A|\\
    &\geq \|x_B^*\|^2 \left[\frac{1-\delta}{2} - \frac{\delta^2}{1-\delta} - \frac{\beta}{\sqrt{1-\delta}}\right] + \lambda|A|,
  \end{aligned}
\end{equation*}
where the last line follows from $\epsilon < \beta\|x_B^*\|$, in view of Assumption \ref{ass:epsilon}.
Appealing again to Assumption \ref{ass:epsilon}, $\epsilon^2/2\leq \beta^2\min_{i\in A^*}|x_i^*|^2/2
\leq \beta^2\|x^*_B\|^2/2$. Next it follows from the assumption $ \beta\leq (1-\delta-\delta^2)/4$ that the inequality
\begin{equation*}
  \begin{aligned}
    \beta^2 + \frac{\beta}{\sqrt{1-\delta}} & \leq \frac{\beta^2 +\beta}{\sqrt{1-\delta}} \leq \frac{2\beta}{1-\delta}\\
      & \leq \frac{1-2\delta-\delta^2}{2(1-\delta)} = \frac{1-\delta}{2} - \frac{\delta^2}{1-\delta}
  \end{aligned}
\end{equation*}
holds. This together with the definition of $\xi$ yields $\xi>\epsilon^2/2$. Further, the choice of $\lambda \in
(\epsilon^2/2,\xi)$ implies
\begin{equation*}
J_\lambda(x) - J_\lambda(x^o) \geq \|x_B^*\|^2 \left[\frac{1-\delta}{2} - \frac{\delta^2}{1-\delta} - \frac{\beta}{\sqrt{1-\delta}} - \frac{1}{2}\beta^2\right] - \lambda |B| > 0,
\end{equation*}
which again leads to a contradiction. This completes the proof of the theorem.
\end{proof}

\begin{proposition}\label{pro:eqv}
Let the conditions in Theorem \ref{thm:global} hold. Then the oracle solution $x^o$ is a minimizer of
\eqref{l0noise}. Moreover, the support to any solution of problem \eqref{l0noise} is $A^*$.
\end{proposition}
\begin{proof}
First we observe that there exists a solution $\bar{x}$ to problem \eqref{l0noise} with $|\mbox{supp}
(\bar{x})|\leq  T$ by noticing that the true solution $x^*$ satisfies $\|\Psi x^* - y\|\leq \epsilon$ and $\|x^*\|_0\leq T$. Clearly, for any
minimizer $\bar{x}$ to problem \eqref{l0noise} with support $|A|\leq T$, then $\Psi^\dag_A y$ is also
a minimizer with $\|\Psi \Psi_A^\dag y - y\|\leq \|\Psi \bar{x} - y\|$. Now if there is a minimizer $\bar x$ with $A \neq A^*$, by
repeating the arguments in the proof of Theorem \ref{thm:global}, we deduce
\begin{equation*}
  \tfrac{1}{2}\|\Psi \Psi_A^\dag y - y \|^2 + \lambda \|\Psi_A^\dag y\|_0 = J_\lambda(\Psi_A^\dag y) > J_\lambda(x^o)=\tfrac{1}{2}\epsilon^2 + \lambda T \Rightarrow \|\Psi \bar{x} - y\| > \epsilon,
\end{equation*}
which leads a contradiction to the assumption that $\bar x$ is a minimizer to problem \eqref{l0noise}. Hence, any minimizer of \eqref{l0noise} has a
support $A^*$, and thus the oracle solution $x^o$ is a minimizer.
\end{proof}

\begin{remark}\label{rem:eqv}
Due to the nonconvex structure of problem \eqref{l0noise}, the equivalence between
problem \eqref{l0noise} and its ``Lagrange'' version \eqref{equ:l0reg} is generally
not clear. However under certain assumptions, their equivalence can be obtained,
cf. Theorem \ref{thm:global} and Proposition \ref{pro:eqv}. Further, we note that very recently, the
equivalence between \eqref{equ:l0reg} and the following constrained sparsity problem
\begin{equation*}
  \min \|\Psi x - y\| \quad \mbox{subject to }\quad \|x\|_0 \leq T
\end{equation*}
was discussed in \cite{Nikolova:2014}.
\end{remark}

\section{Primal-dual active set method with continuation}\label{sec:alg}
In this section, we present the primal-dual active set with continuation (PDASC) algorithm, and
establish its finite step convergence property.
\subsection{The PDASC algorithm}
The PDASC algorithm combines the strengthes of the PDAS algorithm \cite{JiaoJinLu:2013} and the
continuation technique. The complete procedure is described in Algorithm \ref{alg:pdasc}. The
PDAS algorithm (the inner loop) first determines the active set from the primal and dual
variables, then update the primal variable by solving a least-squares problem on the active
set, and finally update the dual variable explicitly. It is well known that for convex optimization problems
the PDAS algorithm can be interpreted as the semismooth Newton method \cite{ItoKunisch:2008}.
Thus the algorithm merits a local superlinear
convergence, and it reaches convergence with a good initial guess. In contrast, the continuation
technique on the regularization parameter $\lambda$ allows one to control the size of the active
set $A$, and thus the active set of the coordinatewise minimizer lies within the true active
set $A^*$. For example, for the choice of the parameter $\lambda_0\geq \|\Psi^t y\|^2_{\ell^\infty}/2$,
$x(\lambda_0)=0$ is the unique global minimizer to the function $J_{\lambda_0}$, and the active set $A$ is empty.

\begin{algorithm}[hbt!]
   \caption{Primal dual active set with continuation (PDASC) algorithm}\label{alg:pdasc}
   \begin{algorithmic}[1]
     \STATE Set $\lambda_0 \geq\frac{1}{2} \|\Psi^t y\|^2_{\ell^\infty}$, ${A}(\lambda_0) = \emptyset$,
        $x(\lambda_0) =0$ and $d(\lambda_0)= \Psi^t y$, $\rho \in (0,1)$, $J_{max}\in \mathbb{N}$.
     \FOR {$k=1,2,...$}
     \STATE Let $\lambda_k = \rho\lambda_{k-1}$, ${A}_0 = {A}(\lambda_{k-1})$, $(x^0,d^0) = (x(\lambda_{k-1}), d(\lambda_{k-1}))$.
     \FOR {$j=1,2,..., J_{max}$}
     \STATE Compute the active and inactive sets ${A}_j$ and ${I}_j$:
         \begin{equation*}
            {A}_j = \left\{i: |x^{j-1}_i+ d^{j-1}_i| > \sqrt{2\lambda_k}\right\}\quad\mbox{and}\quad {I}_j = {A}_j^c.
         \end{equation*}
     \STATE Check stopping criterion ${A}_j = {A}_{j-1}$. 
     \STATE Update the primal and dual variables $x^j$ and $d^j$ respectively by
       \begin{equation*}
        \left\{
          \begin{array}{l}
           x_{{I}_j}^{j} = 0, \\[1.2ex]
           \Psi_{{{A}}_j}^t \Psi_{{{A}}_j} x_{{{A}}_j}^{j} = \Psi_{{{A}}_j}^t y,  \\[1.2ex]
           d^{j} = \Psi^t(\Psi x^j - y).
          \end{array}\right.
       \end{equation*}
     \ENDFOR
     \STATE {blue}Set $\widetilde{j}=\min(J_{max},j)$, and ${A}(\lambda_k) = \left\{i: |x^{\widetilde{j}}_i+ d^{\widetilde{j}}_i| > \sqrt{2\lambda_k}\right\}$ and  $(x(\lambda_{k}),d(\lambda_{k})) = (x^{\widetilde{j}},d^{\widetilde{j}})$.
     \STATE Check stopping criterion: $\|\Psi x(\lambda_k) - y\| \leq \epsilon$.
     \ENDFOR
   \end{algorithmic}
\end{algorithm}

In the algorithm, there are a number of free parameters: the starting value $\lambda_0$ for
the parameter $\lambda$, the decreasing factor $\rho\in(0,1)$ (for $\lambda$), and the maximum number
$J_{max}$ of iterations for the inner PDAS loop. Further, one needs to set the stopping criteria
at lines 6 and 10. Below we discuss their choices.

The choice of initial value $\lambda_0$ is not important. For any choice $\lambda_0\geq \|\Psi^t y
\|^2_{\ell^\infty}/2$, $x=0$ is the unique global minimizer, and $A=\emptyset$. Both the decreasing
factor $\rho$ and the iteration number $J_{max}$ affect the accuracy and efficiency of the algorithm:
Larger $\rho$ and $J_{max}$ values make the algorithm have better exact support recovery probability
but take more computing time. Numerically, $\rho$ is determined by the number of grid points
for the parameter $\lambda$. Specifically, given an initial value $\lambda_0 \geq  \|\Psi^t y\|^2_{\ell^\infty}/2$
and a small constant $\lambda_{min}$, e.g., $\mbox{1e-15}\lambda_0$, the interval $[\lambda_{min},\lambda_0]$ is divided
into $N$ equally distributed subintervals in the logarithmic scale. A large $N$ implies a large decreasing
factor $\rho$. The choice $J_{max}=1$ generally works well, which is also covered in the convergence
theory in Theorems \ref{thm:mc} and \ref{thm:rip} below.

The stopping criterion for each $\lambda$-problem in Algorithm \ref{alg:pdasc} is
either $A_j = A_{j-1}$ or $j =  J_{max}$, instead of the standard criterion $A_j = A_{j-1}$ for
active set type algorithms. The condition $j = J_{max}$ is very important for nonconvex problems.
This is motivated by the following empirical observation: When the true signal $x^*$ does not have a
strong decay property, e.g., $0$-$1$ signal, the inner PDAS loop (for each $\lambda$-problem) may
never reach the condition ${A}_j = {A}_{j-1}$ within finite steps; see the example below.

\begin{exam}
In this example, we illustrate the convergence of the PDAS algorithm. Let $-1<\mu<0$, ${A}^* = \{1,2\}$, and
\begin{equation*}
   \Psi_1 = \frac{1}{\sqrt{1+\mu^2}}(1,\mu,0,...,0)^t,\ \ \Psi_2 = \frac{1}{\sqrt{1+\mu^2}}(\mu, 1, 0,...,0)^t, \ \ x^*_1 = x^*_2 = 1.
\end{equation*}
In the absence of data noise $\eta$, the data $y$ is given by
\begin{equation*}
  y = \frac{1}{\sqrt{1+\mu^2}}(1+\mu, 1+\mu,0,...,0)^t.
\end{equation*}
Now we let $\sqrt{2\lambda}\in (\frac{(1+\mu)^2}{1+\mu^2},\frac{(1-\mu^2)^2}{(1+\mu^2)^2})$, the initial
guess ${A}_1 = \{1\}$. Then direct computation yields
\begin{equation*}
  \begin{aligned}
    x^1 &= \frac{1}{1+\mu^2}((1+\mu)^2, 0)^t,\\
    y - \Psi x^1 & = \frac{1 - \mu^2}{(\sqrt{1+ \mu^2})^3}(-\mu, 1,0,...,0)^t,\\
    d^1 &= \frac{1}{(1+\mu^2)^2}(0, (1-\mu^2)^2)^t.
  \end{aligned}
\end{equation*}
Hence $d^1_2 > \sqrt{2\lambda}> x^1_1$, and ${A}_2 = \{2\}$. Similarly, we have
${A}_3 = \{1\} = {A}_1$, which implies that the algorithm simply alternates between the
two sets $\{1\}$ and $\{2\}$ and will never reach the stopping condition ${A}_k = {A}_{k+1}$.
\end{exam}

The stopping condition at line 10 of Algorithm \ref{alg:pdasc} is a discrete analogue of
the discrepancy principle. This rule is well established in the inverse problem community
for selecting an appropriate regularization parameter \cite{ItoJin:2014}.
The rationale behind the rule is that one cannot expect the reconstruction to be more accurate
than the data accuracy in terms of the discrepancy. In the PDASC algorithm, if the active set
is always contained in the true active set $A^*$ throughout the iteration, then the discrepancy principle can
always be satisfied for some $\lambda_k$, and the solution $x(\lambda_k)$ resembles closely the oracle solution $x^o$.

\subsection{Convergence analysis}

Now we discuss the convergence of Algorithm \ref{alg:pdasc}. We shall discuss
the cases of the MIP and RIP conditions separately. The general proof strategy is as follows.
It essentially relies on the control of the active set during the iteration and the certain
monotonicity relation of the active set $A(\lambda_k)$ (via the continuation technique). In particular, we
introduce two auxiliary sets $G_{\lambda,s_1}$ and $G_{\lambda,s_2}$, cf. \eqref{equ:sets} below, to
precisely characterize the evolution of the active set $A$ during the PDASC iteration.

First we consider the MIP case.
We begin with an elementary observation: under the assumption $\nu<(1-2\beta)/(2T-1)$
of the mutual coherence parameter $\nu$, there holds $(2T-1)\nu + 2\beta <1$.
\begin{lemma}\label{lem:mc-ss}
If $\nu < (1-2\beta)/(2T - 1)$, then for any $\rho \in ( ((2T-1)\nu+2\beta)^2,1)$ there exist $s_1,s_2\in (1/(1-T\nu
+ \nu -\beta),1/(T\nu + \beta))$, $s_1>s_2$, such that $s_2 = 1+ (T\nu - \nu +\beta)s_1$ and $\rho=s_2^2/s_1^2$.
\end{lemma}
\begin{proof}
By the assumption $v<(1-2\beta)/(2T-1)$, $T\nu + \beta < 1 - T\nu + \nu -\beta$. Hence for any $s_1\in (1/(1-T\nu + \nu -\beta),1/(T\nu + \beta))$,
there holds
\begin{equation*}
  s_1>1+(T\nu -\nu +\beta)s_1 \quad \mbox{and}\quad 1+(T\nu-\nu+\beta)s_1>\frac{1}{1-T\nu+\nu-\beta},
\end{equation*}
i.e.,
\begin{equation*}
\frac{1}{T\nu + \beta} >s_1 > 1+ (T\nu - \nu +\beta)s_1 > \frac{1}{1-T\nu + \nu -\beta}.
\end{equation*}
Upon letting $s_2 = 1+ (T\nu - \nu +\beta)s_1$, we deduce
\begin{equation*}
\frac{1}{T\nu + \beta} >s_1 > s_2 > \frac{1}{1-T\nu + \nu -\beta}.
\end{equation*}
Now the monotonicity of the function $f(s_1)=s_2/s_1$ over the interval $(1/(1-T\nu
+\nu -\beta),1/(T\nu + \beta))$, and the identities
\begin{equation*}
  \begin{aligned}
    &\frac{1 + (T\nu - \nu +\beta)/(T\nu + \beta)}{1/(T\nu + \beta)} = (2T-1)\nu + 2\beta,\\
    &\frac{1 + (T\nu - \nu +\beta)/(1-T\nu + \nu -\beta)}{1/(1-T\nu + \nu -\beta)} = 1,\\
  \end{aligned}
\end{equation*}
imply that there exists an $s_1$ in the internal such that $s_2/s_1 = \sqrt{\rho}$ for any $\rho \in ( ((2T-1)\nu+2\beta)^2,1)$.
\end{proof}

Next for any $\lambda>0$ and $s>0$, we denote by
\begin{equation}\label{equ:sets}
  G_{\lambda,s} \triangleq \left\{i: |x_i^*| \geq \sqrt{2\lambda}s\right\}.
\end{equation}
The set $G_{\lambda,s}$ characterizes the true sparse signal $x^*$ (via level sets). The lemma
below provides an important monotonicity relation on the active set $A_k$ during the iteration,
which is essential for showing the finite step convergence of the algorithm in Theorem \ref{thm:global} below.
\begin{lemma}\label{lem:mcactiveset}
Let Assumption \ref{ass:epsilon} hold, $\nu < (1-2\beta)/({2T - 1})$, $\rho \in ( ((2T-1)\nu+2\beta)^2,1)$, and
$s_1$ and $s_2$ be defined in Lemma \ref{lem:mc-ss}. If $G_{\lambda,s_1}\subseteq A_k\subseteq A^*$,
then $G_{\lambda,s_2}\subseteq A_{k+1}\subseteq A^*$.
\end{lemma}
\begin{proof}
Let $A = A_k$, $B = A^*\backslash A$. By Lemma \ref{lem:est-onestep}, we have
\begin{equation*}
  \begin{aligned}
    |x_i| &\geq |x_i^*| - \|\bar{x}_A\|_{\ell^\infty} \geq |x_i^*| - \frac{|B|\nu \|x_B^*\|_{\ell^\infty} + \epsilon}{1 - (|A|-1)\nu}, \quad \forall i\in A,\\
    |d_j| &\leq |B|\nu \left(1 + \frac{|A|\nu }{1- (|A| -1)\nu}\right) \|x_B^*\|_{\ell^\infty} + \epsilon \left(1 + \frac{|A|\nu }{1- (|A| -1)\nu}\right), \quad \forall \,\, j\in I^*,\\
    |d_i|&\geq |x_i^*|+ \nu\|x_B^*\|_{\ell^\infty} - |B|\nu \left(1 + \frac{|A|\nu }{1- (|A| -1)\nu}\right) \|x_B^*\|_{\ell^\infty} - \epsilon \left(1 + \frac{|A|\nu }{1- (|A| -1)\nu}\right), \quad \forall \,\, i\in B.
  \end{aligned}
\end{equation*}
Using the fact $\epsilon\leq \beta\min_{i\in A^*}|x_i^*|\leq \beta\|x_B^*\|_{\ell^\infty}$ from Assumption \ref{ass:epsilon}
and the trivial inequality $\frac{|B|\nu + \beta}{1- T\nu +\nu + |B|\nu} \leq \frac{T\nu + \beta}{1+\nu}$,
we arrive at
\begin{equation*}
  \begin{aligned}
     |B|\nu \left(1 + \frac{|A|\nu }{1- (|A| -1)\nu}\right) &\|x_B^*\|_{\ell^\infty} + \epsilon \left(1 + \frac{|A|\nu }{1- (|A| -1)\nu}\right) \\
     \leq& \frac{|B|\nu + \beta}{1- T\nu +\nu + |B|\nu}(1+\nu)\|x_B^*\|_{\ell^\infty}
      \leq (T\nu +\beta)\|x_B^*\|_{\ell^\infty}.
  \end{aligned}
\end{equation*}
Consequently,
\begin{equation*}
  \begin{aligned}
    |d_j|&\leq (T\nu + \beta)\|x_B^*\|_{\ell^\infty}, \quad \forall j\in I^*, \\
    |d_i|&\geq |x_i^*| - (T\nu - \nu +\beta)\|x_B^*\|_{\ell^\infty}, \quad \forall i\in B.
  \end{aligned}
\end{equation*}
It follows from the assumption $G_{\lambda,s_1}\subseteq A =  A_k$ that $\|x^*_B\|_{\ell^\infty} < s_1\sqrt{2\lambda}$. Then for all $j\in I^*$, we
have
\begin{equation*}
  |d_j| < s_1(T\nu + \beta)\sqrt{2\lambda} < \sqrt{2\lambda},
\end{equation*}
i.e., $j\in I_{k+1}$. This shows $A_{k+1}\subseteq A^*$. For any $i\in I \cap G_{\lambda,s_2}$, we have
\begin{equation*}
  |d_i| > s_2\sqrt{2\lambda} - (T\nu - \nu +\beta)s_1 \sqrt{2\lambda} \geq \sqrt{2\lambda},
\end{equation*}
This implies $i\in A_{k+1}$ by \eqref{equ:con2}. It remains to show that for any $i\in A
\cap G_{\lambda,s_2}$, $i\in A_{k+1}$. Clearly, if $A = \emptyset$, the assertion holds. Otherwise
\begin{equation*}
  \begin{aligned}
   |x_i| & \geq |x_i^*| - \frac{|B|\nu +\beta}{1 - (|A|-1)\nu}\|x_B^*\|_{\ell^\infty} \\
     & > s_2\sqrt{2\lambda} - (T\nu - \nu +\beta)s_1 \sqrt{2\lambda} \geq \sqrt{2\lambda},
  \end{aligned}
\end{equation*}
where the last line follows from the elementary inequality
\begin{equation*}
\frac{|B|\nu  + \beta}{1 - (|A|-1)\nu} \leq T\nu - \nu +\beta.
\end{equation*}
This together with \eqref{equ:con2} also implies $i\in A_{k+1}$. This concludes the proof of the lemma.
\end{proof}

Now we can state the convergence result.
\begin{theorem}\label{thm:mc}
Let Assumption \ref{ass:epsilon} hold, and $\nu < (1-2\beta)/({2T - 1})$. Then for any $\rho \in (((2T-1)\nu + 2\beta)^2,1)$, Algorithm
\ref{alg:pdasc} converges in finite steps.
\end{theorem}
\begin{proof}
For each $\lambda_k$-problem, we denote by $A_{k,0}$ and $A_{k,\diamond}$ the active set for the
initial guess and the last inner step (i.e., $A(\lambda_k)$ in Algorithm \ref{alg:pdasc}),
respectively. Now with $s_1$ and $s_2$ from Lemma \ref{lem:mc-ss}, there holds $G_{\lambda,s_1}\subset G_{\lambda,s_2}$,
and using Lemma \ref{lem:mcactiveset}, for any index $k$ before the stopping criterion at line 10 of Algorithm
\ref{alg:pdasc} is reached, there hold
\begin{equation}\label{equ:G}
  G_{\lambda_k,s_1} \subseteq A_{k,0}\quad\mbox{and}\quad G_{\lambda_k,s_2}\subseteq A_{k,\diamond}.
\end{equation}
Note that for $k=0$, $G_{\lambda_0,s_2} = \emptyset$ and thus the assertion holds. To see this,
it suffices to check $\|x^*\|_{\ell^\infty} < s_2 \|\Psi^t y\|_{\ell^\infty}$. By Lemma
\ref{lem:mc} and the inequality $s_2 > 1/({1-T\nu + \nu -\beta})$ we obtain that
\begin{equation*}
  \begin{aligned}
    \|\Psi^t y\|_{\ell^\infty} &\geq \|\Psi_{A^*}^t\Psi_{A^*} x^*_{A^*}\|_{\ell^\infty} -\| \Psi^t \eta\|_{\ell^\infty} \\
     &\geq (1 - (T-1)\nu)\|x^*\|_{\ell^\infty} - \epsilon >\|x^*\|_{\ell^\infty}/s_2.
  \end{aligned}
\end{equation*}
Now for $k>0$, it follows by mathematical induction and the relation $A_{k,\diamond} = A_{k+1,0}$.
It follows from \eqref{equ:G} that during the iteration, the active set $A_{k,\diamond}$ always
lies in $A^*$. Further, for $k$ sufficiently large, by Lemma \ref{lem:mon}, the stopping
criterion at line 10 must be reached and thus the algorithm terminates; otherwise
\begin{equation*}
  A^* \subseteq G_{\lambda_k,s_1},
\end{equation*}
then the stopping criterion at line 10 is satisfied, which leads to a contradiction.
\end{proof}

Next we turn to the convergence of Algorithm \ref{alg:pdasc} under the RIP condition. Let $1 - (2\sqrt{T} + 1)\delta > 2\beta$, an argument analogous to Lemma
\ref{lem:mc-ss} implies that for any $\sqrt{\rho}\in ((2\delta\sqrt{T} + 2\beta)/({1-\delta}),1)$
there exist $s_1$ and $s_2$ such that
\begin{equation}\label{equ:s1s2}
\frac{1-\delta}{\delta \sqrt{T} + \beta}> s_1 > s_2 >\frac{1-\delta}{1 - \delta -\delta\sqrt{T} - \beta}, \quad   s_2 = 1+\frac{\delta\sqrt{T} + \beta}{1-\delta}s_1 , \quad \frac{s_2}{s_1} = \sqrt{\rho}.
\end{equation}

The next result is an analogue of Lemma \ref{lem:mcactiveset}.
\begin{lemma}
Let Assumption \ref{ass:epsilon} hold, $\delta \triangleq \delta_{T+1} \leq (1-2\beta)/(2\sqrt{T}+1)$, and $\sqrt{\rho}\in ((2\delta\sqrt{T} + 2\beta)/({1-\delta}),1)$. Let $s_1$ and $s_2$ are defined by \eqref{equ:s1s2}.
If $G_{\lambda,s_1}\subseteq A_k\subseteq A^*$, then $G_{\lambda,s_2}\subseteq A_{k+1}\subseteq A^*$.
\end{lemma}
\begin{proof}
Let $A = A_k$, $B = A^*\backslash A$. Using the notation in Lemma \ref{lem:est-onestep}, we have
\begin{equation*}
  \begin{aligned}
    |x_i| &\geq |x_i^*| - \|\bar{x}_A\| \geq |x_i^*| - \frac{\delta \|x_B^*\| + \epsilon}{1 -\delta} , \quad \forall i\in A,\\
    |d_j| &\leq \delta\|x_B^*\| + \epsilon + \delta\|\bar{x}_A\| \leq \frac{\delta \|x_B^*\| + \epsilon}{1-\delta}, \quad \forall  j\in I^*,\\
    |d_i|&\geq |x_i^*| - \delta\|x_B^*\| - \epsilon - \delta\|\bar{x}_A\| \geq |x_i^*|-\frac{\delta \|x_B^*\| + \epsilon}{1-\delta}, \quad \forall i\in B.
  \end{aligned}
\end{equation*}
By the assumption $G_{\lambda,s_1}\subseteq A_k$, we have $\|x_B^*\|_{\ell^\infty} <s_1\sqrt{2\lambda}$.
Now using the relation $s_1 < (1-\delta)/({\delta \sqrt{T} + \beta})$ and Assumption
\ref{ass:epsilon}, we deduce
\begin{equation*}
\frac{\delta \|x_B^*\| + \epsilon}{1-\delta} \leq \frac{\delta\sqrt{T} + \beta}{1 - \delta}\|x_B^*\|_{\ell^\infty} < \sqrt{2\lambda}.
\end{equation*}
Thus for $j\in I^*$, $|d_i|<\sqrt{2\lambda}$, i.e., $A_{k+1}\subset A^*$. Similarly,
using the relations $s_2 =  1+s_1(\delta\sqrt{T} + \beta)/({1-\delta})$ and $s_1>
(1-\delta)/(1 - \delta -\delta\sqrt{T} - \beta)$, we arrive at that for any $i\in G_{\lambda,s_2}$, there holds
\begin{equation*}
    |x_i^*| - \frac{\delta \|x_B^*\| + \epsilon}{1-\delta} > s_2\sqrt{2\lambda} - \frac{\delta\sqrt{T}
    + \beta}{1 - \delta}s_1\sqrt{2\lambda} = \sqrt{2\lambda}. 
\end{equation*}
This implies that for $i\in G_{\lambda,s_2}\cap A$, $|x_i|>\sqrt{2\lambda}$, and for $i\in G_{\lambda,s_2}
\cap I$, $|d_i|>\sqrt{2\lambda}$. Consequently, \eqref{equ:con2} yields the desired relation $(G_{\lambda,s_2}\cap A)
\subseteq A_{k+1}$, and this concludes the proof of the lemma.
\end{proof}

Now we can state the convergence of Algorithm \ref{alg:pdasc} under the RIP assumption.
The proof is similar to that for Theorem \ref{thm:mc}, and hence omitted.
\begin{theorem}\label{thm:rip}
Let Assumption \ref{ass:epsilon} hold, and $\delta \triangleq \delta_{T+1} \leq (1-2\beta)/({2\sqrt{T}+1})$.
Then for any $\sqrt{\rho} \in \left((2\delta\sqrt{T} + 2\beta)/({1-\delta}),1\right)$, Algorithm \ref{alg:pdasc} converges in finite steps.
\end{theorem}

\begin{remark}
Theorems \ref{thm:mc} and \ref{thm:rip} indicate that Algorithm \ref{alg:pdasc} converges
in finite steps, and the active set $A(\lambda_k)$ remains a subset of the true active set $A^*$.
\end{remark}

\begin{corollary}
Let the assumptions in Theorem \ref{thm:global} hold. Then Algorithm \ref{alg:pdasc} terminates at the oracle solution $x^o$.
\end{corollary}
\begin{proof}
First, we note the monotonicity relation $A(\lambda_k)\subset A^*$ before the stopping criterion
at line 10 of Algorithm \ref{alg:pdasc} is reached. For any $A\subsetneq A^*$, let $x = \Psi_A^\dag
y$. Then by the argument in the proof of Theorem \ref{thm:global}, we have
\begin{equation*}
J_\lambda(x) = \tfrac{1}{2}\|\Psi x - y\|^2 + \lambda |A| > \tfrac{1}{2}\epsilon^2 + \lambda T \Rightarrow \|\Psi x - y\| > \epsilon,
\end{equation*}
which implies that the stopping criterion at line 10 in Algorithm \ref{alg:pdasc} cannot be
satisfied until the oracle solution $x^o$ is reached.
\end{proof}

\subsection{Connections with other algorithms}

Now we discuss the connections of Algorithm \ref{alg:pdasc} with two existing greedy
methods, i.e., orthogonal matching pursuit (OMP) and hard thresholding pursuit (HTP).

\paragraph{Connection with the OMP.} To prove the convergence of Algorithm \ref{alg:pdasc}, we require
either the MIP condition ($\nu < (1-2\beta)/({2T - 1})$) or the RIP condition ($
\delta_{T+1} \leq (1-2\beta)/({2\sqrt{T}+1})$) on the sensing matrix $\Psi$. These
assumptions have been used to analyze the OMP before: MIP appeared in
\cite{CaiWang:2011} and RIP appeared in \cite{HuangZhu:2011}. Further, for the OMP, the MIP
assumption is fairly sharp, but the RIP assumption can be improved \cite{Zhang:2011,MoShen:2012}.
Our convergence analysis under these assumptions, unsurprisingly, follows the same line of thought as that
for the OMP, in that we require the active set $A(\lambda_k)$ always lies in the true active set
$A^*$ during the iteration. However, we note that this requirement is unnecessary for the PDASC,
since the active set can move inside and outside the true active set $A^*$ during the iteration. The
numerical examples in section \ref{sec:numeric} below confirm this observation. This makes the PDASC much
more flexible than the OMP.

\paragraph{Connection with the HTP.} Actually, the HTP due to Foucart \cite{Foucart:2011} can be viewed
a primal-dual active set method in the $T$-version, i.e., at each iteration, the active set is chosen
by the first $T$-component for both primal and dual variables. This is equivalent to a variable regularization
parameter $\lambda$, where $\sqrt{2\lambda}$ is set to the $T$-th components of $|x^k| + |d^k|$ at each iteration.
Naturally, one can also apply a continuation strategy on the parameter $T$.

\section{Numerical tests}\label{sec:numeric}
In this section we present numerical examples to illustrate the efficiency and accuracy of the proposed PDASC algorithm.
The sensing matrix $\Psi$ is of size $n\times p$, the true solution $x^*$ is a $T$-sparse
signal with an active set $A^*$. The dynamical range $R$ of the true signal $x^*$ is defined by $R = M/m$,
with $M = \max\{|x^{*}_{i}:i\in A^{*}|\}$ and $m= \min\{|x^{*}_{i}|:i\in A^{*}\}=1$.
The data $y$ is generated by
\begin{equation*}
  y = \Psi x^* + \eta,
\end{equation*}
where $\eta$ denotes the measurement noise, with each entry $\eta_i$ following the
Gaussian distribution $N(0,\sigma^2)$ with mean zero and standard deviation
$\sigma$. The exact noise level $\epsilon$ is given by $\epsilon =\|\eta\|_2$.

In Algorithm \ref{alg:pdasc}, we always take $\lambda_{0} = \|\Psi^t y\|_{\ell^\infty}$, and
$\lambda_{min} = \mbox{1e-15}\lambda_{0}$. The choice of the number of grid points $N$ and
the maximum number $J_{max}$ of inner iterations will be specified later.

\subsection{The behavior of the PDASC algorithm}
First we study the influence of the parameters in the PDASC algorithm on the exact recovery probability.
To this end, we fix $\Psi$ to be a $500\times 1000$ random Gaussian matrix, and $\sigma = \mbox{1e-2}$.
All the results are computed based on 100 independent realizations of the problem setup. To this
end, we consider the following three settings:
\begin{enumerate}
\item[(a)] $J_{max} =5$, and varying $N$; see Fig. \ref{fig:grid}(a).
\item[(b)] $N=100$, and varying $J_{max}$; see Fig. \ref{fig:grid}(b).
\item[(c)] $N=100$, $J_{max}=5$, and an approximate noise level $\bar{\epsilon}$; see Fig. \ref{fig:grid}(c).
\end{enumerate}

We observe that the influence of the parameters $N$ and $J_{max}$ is very mild on the exact support
recovery probability. In particular, a reasonably small value for these parameters (e.g. $N=50$, $J_{max} =1$) is
sufficient for accurately recovering the exact active set $A^*$. Unsurprisingly, a very small
value of $N$ can degrade the accuracy of active set recovery greatly, due to insufficient
resolution of the solution path. In practice, the exact noise level $\epsilon$ is not always
available, and often only a rough estimate $\bar{\epsilon}$ is provided. The use of the estimate
$\bar{\epsilon}$ in place of the exact one $\epsilon$ in Algorithm \ref{alg:pdasc} may sacrifice the
recovery probability. Hence it is important to study the sensitivity of Algorithm \ref{alg:pdasc}
with respect to the variation of the parameter $\epsilon$. We observe from Fig. \ref{fig:grid}(c) that the
parameter $\epsilon$ does not affect the recovery
probability much, unless the estimate $\bar{\epsilon}$ is grossly erroneous.

\begin{figure}[htb!]
  \centering
  \begin{tabular}{cc}
   \includegraphics[trim = 1cm 0cm 1cm 0cm, clip=true,width=6cm]{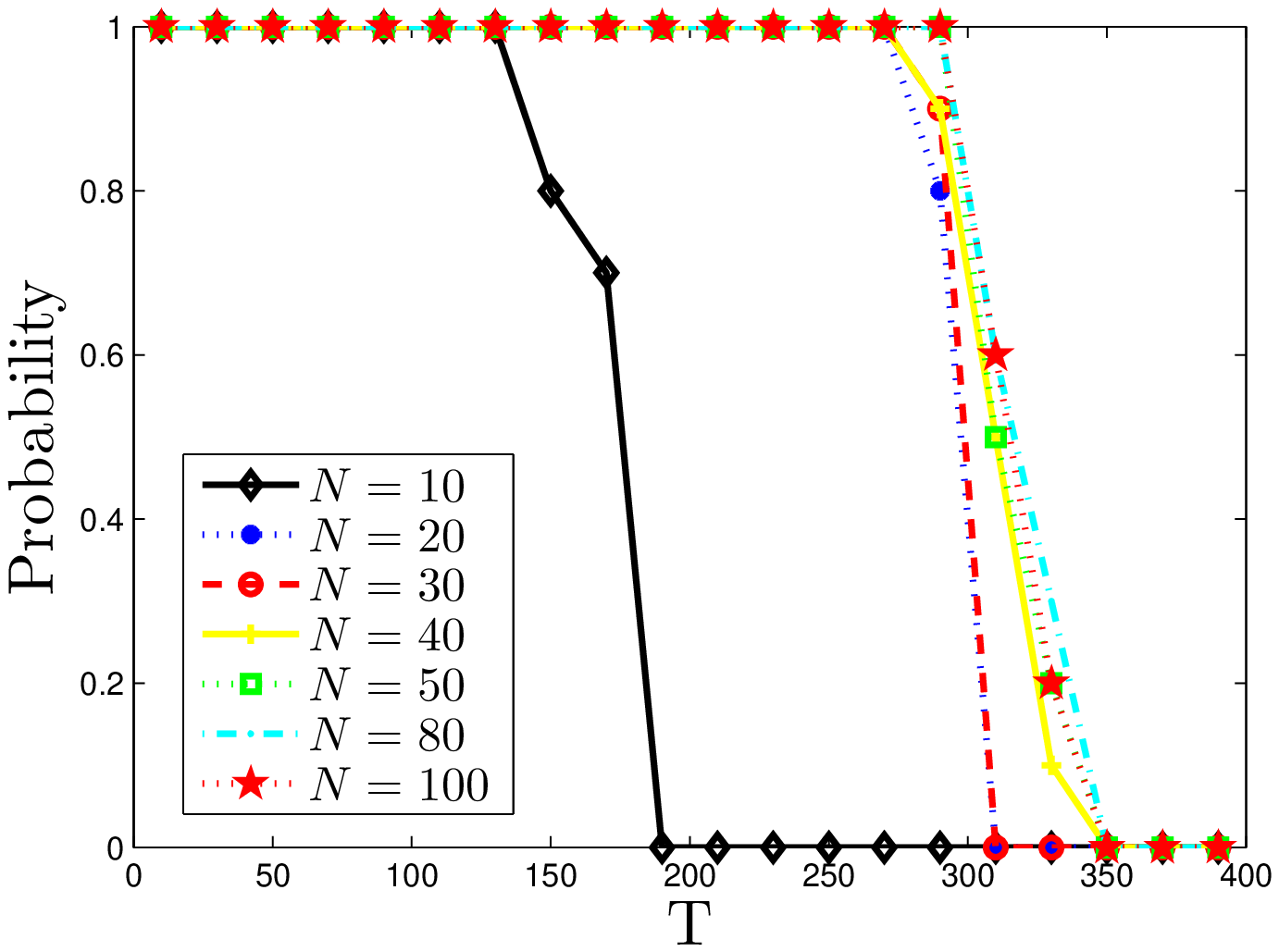} &
   \includegraphics[trim = 1cm 0cm 1cm 0cm, clip=true,width=6cm]{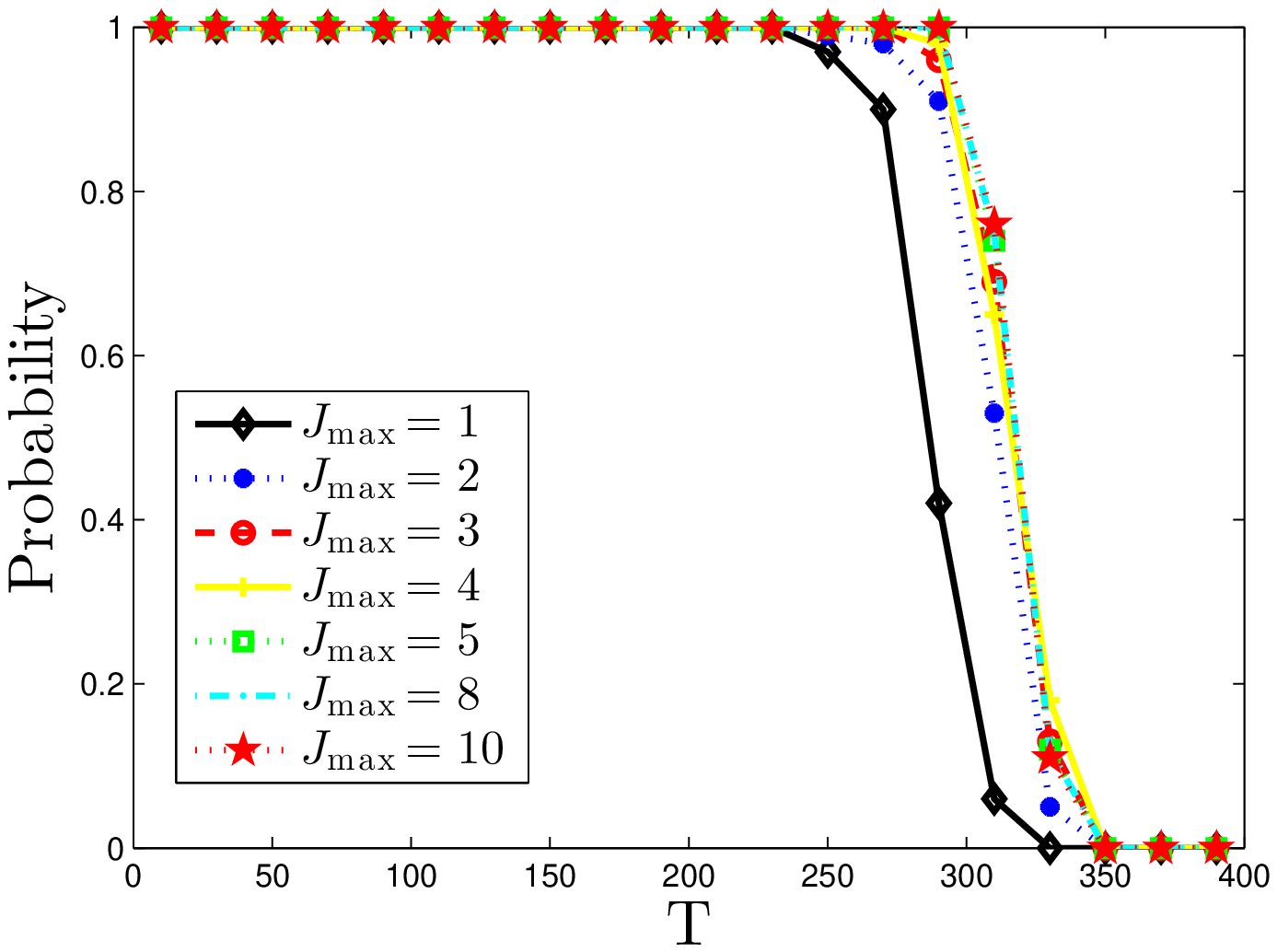}\\
   (a) $N$ & (b) $J_{max}$\\
   \includegraphics[trim = 1cm 0cm 1cm 0cm, clip=true,width=6cm]{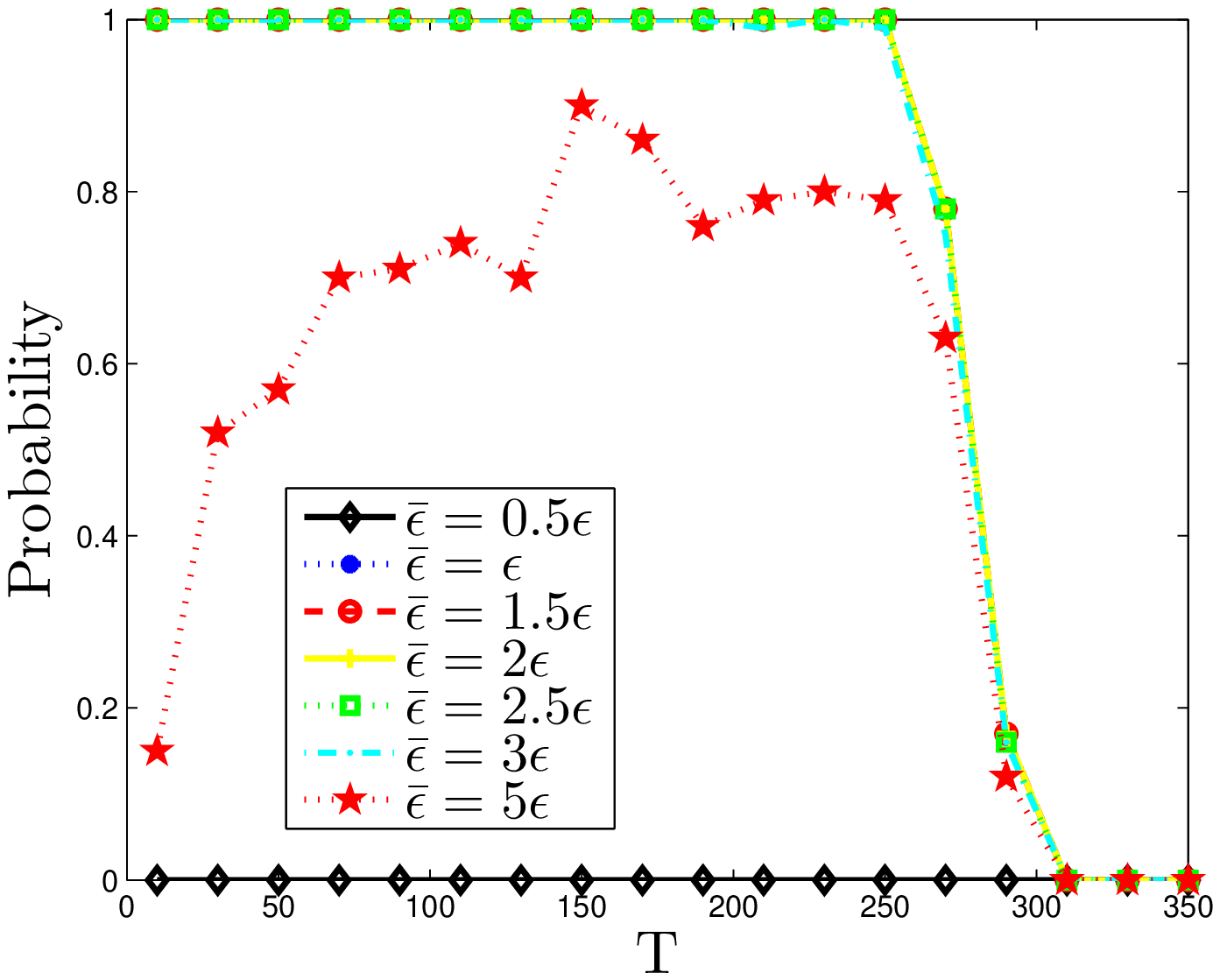}\\
   (c) $\epsilon$\\
  \end{tabular}
   \caption{The influence of the algorithmic parameters ($N$, $J_{max}$ and $\epsilon$)
   on the exact recovery probability.}\label{fig:grid}
\end{figure}

To gain further insight into the PDASC algorithm, in Fig. \ref{fig:acinout},
we show the evolution of the active set (for simplicity let $A_k = A(\lambda_k)$) . It is
observed that the active set $A_k$ can generally move both ``inside'' and ``outside''
of the true active set $A^*$. This is in sharp contrast to the OMP, where the
size of the active set is monotone during
the iteration. The flexible change in the active set might be essential for
the efficiency of the algorithm. This observation is valid for random
Gaussian, random Bernoulli and partial DCT sensing matrices.

\begin{figure}[htb!]
  \centering
  \begin{tabular}{cc}
    \includegraphics[trim = 0cm 0cm 0cm 0cm, clip=true,width=6cm]{{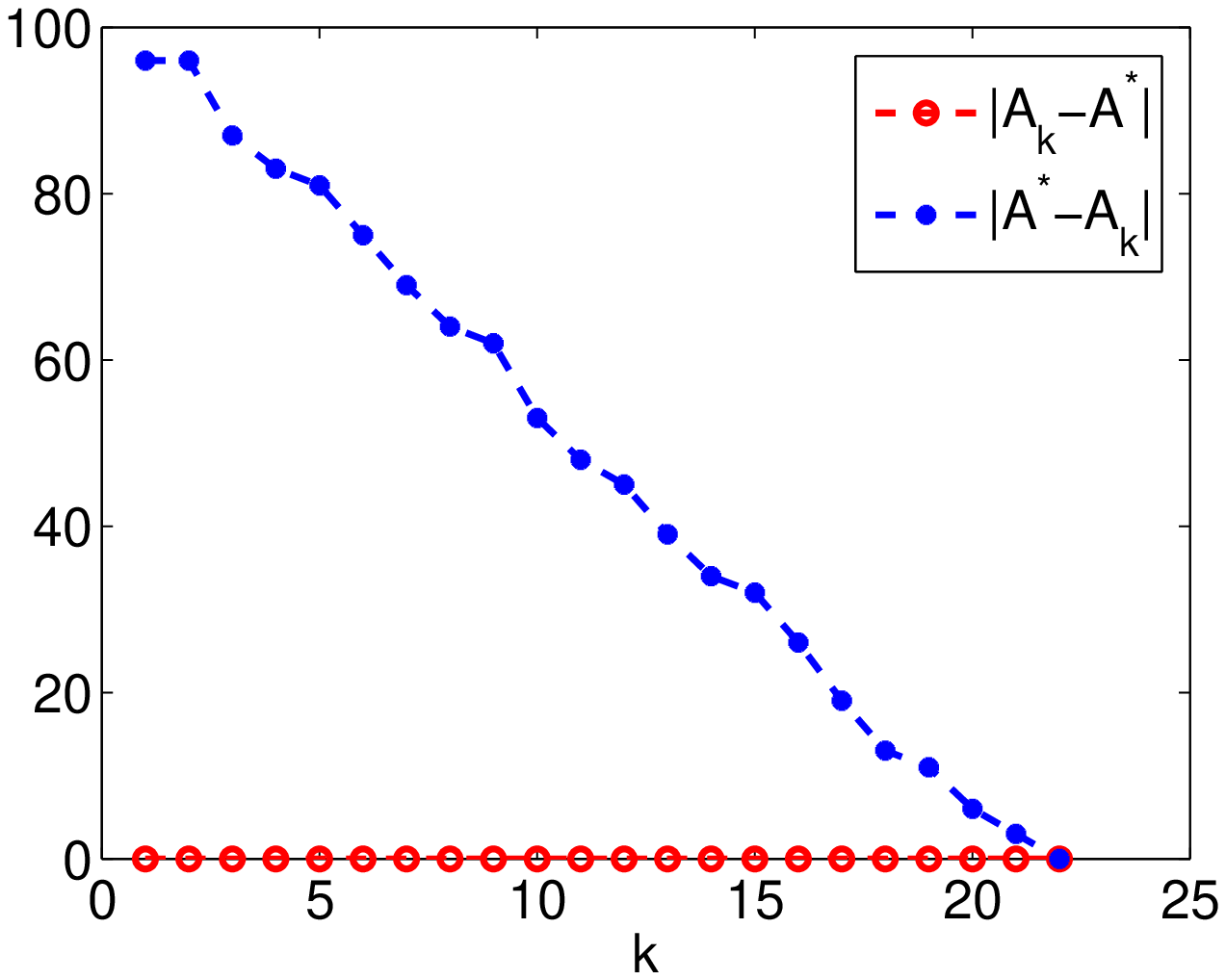}} &\includegraphics[trim = 0cm 0cm 0cm 0cm, clip=true,width=6cm]{{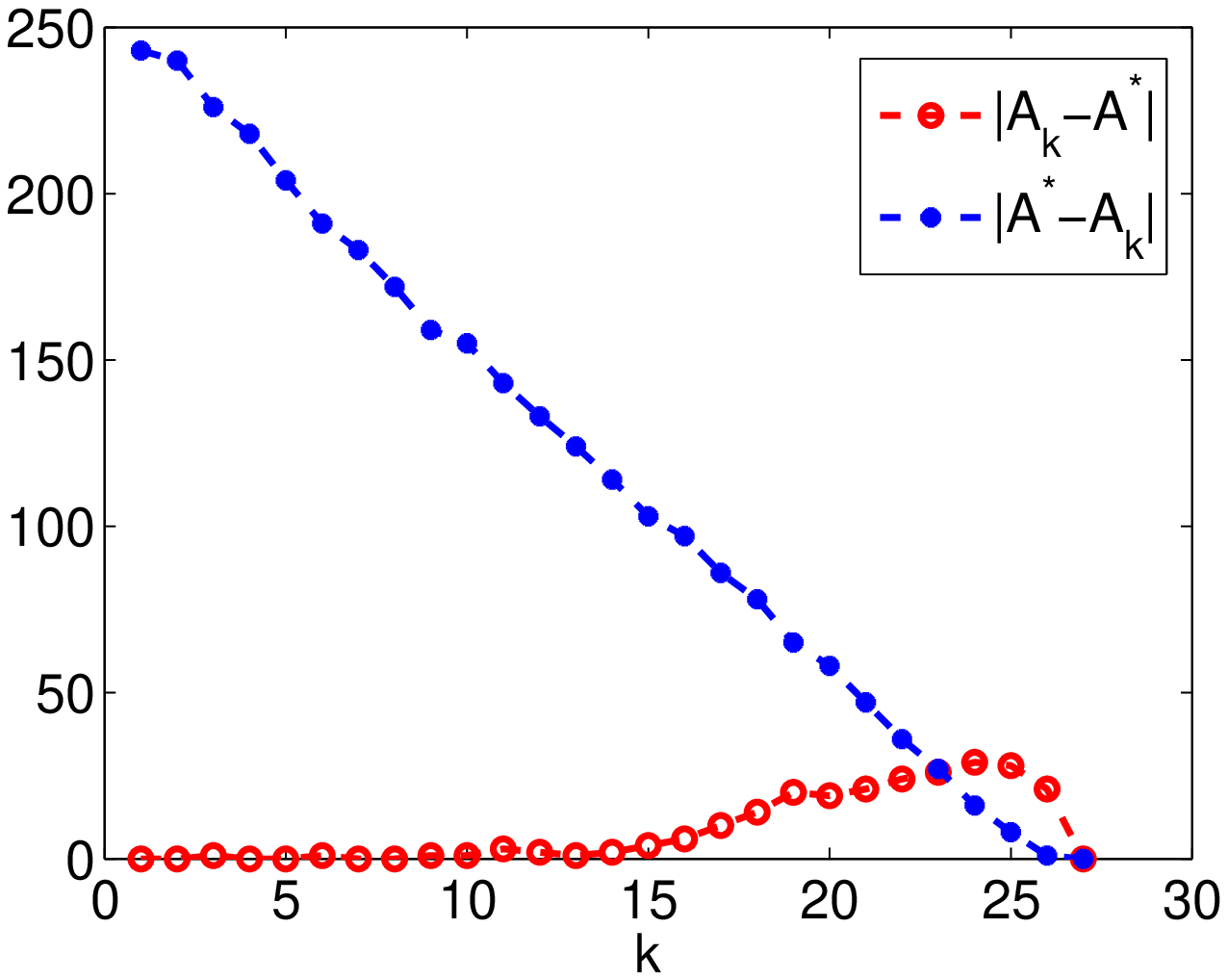}}\\
    (a) random Gaussian, $T=100$  & (b) random Gaussian, $T=250$\\
    \includegraphics[trim = 0cm 0cm 0cm 0cm, clip=true,width=6cm]{{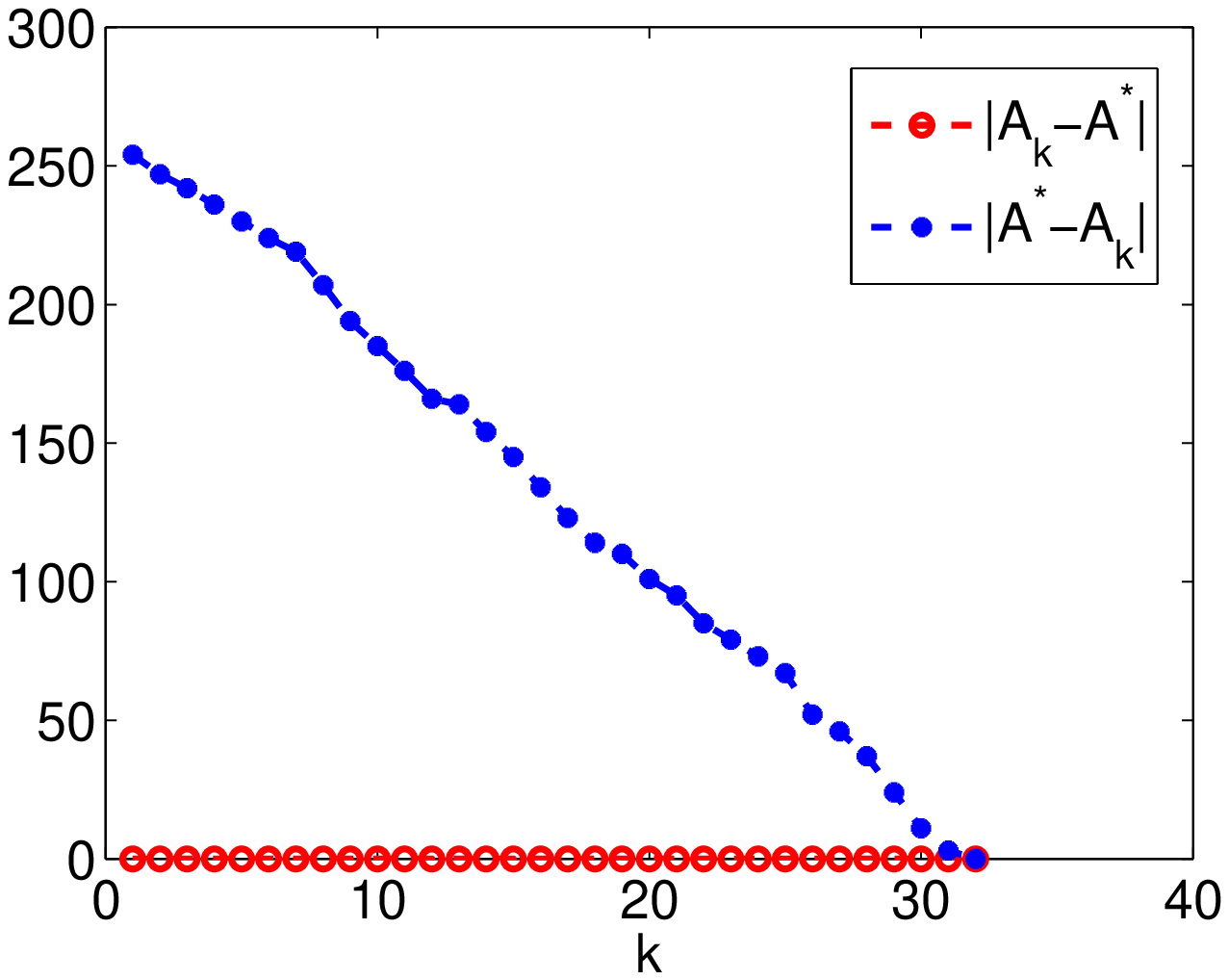}} &\includegraphics[trim = 0cm 0cm 0cm 0cm, clip=true,width=6cm]{{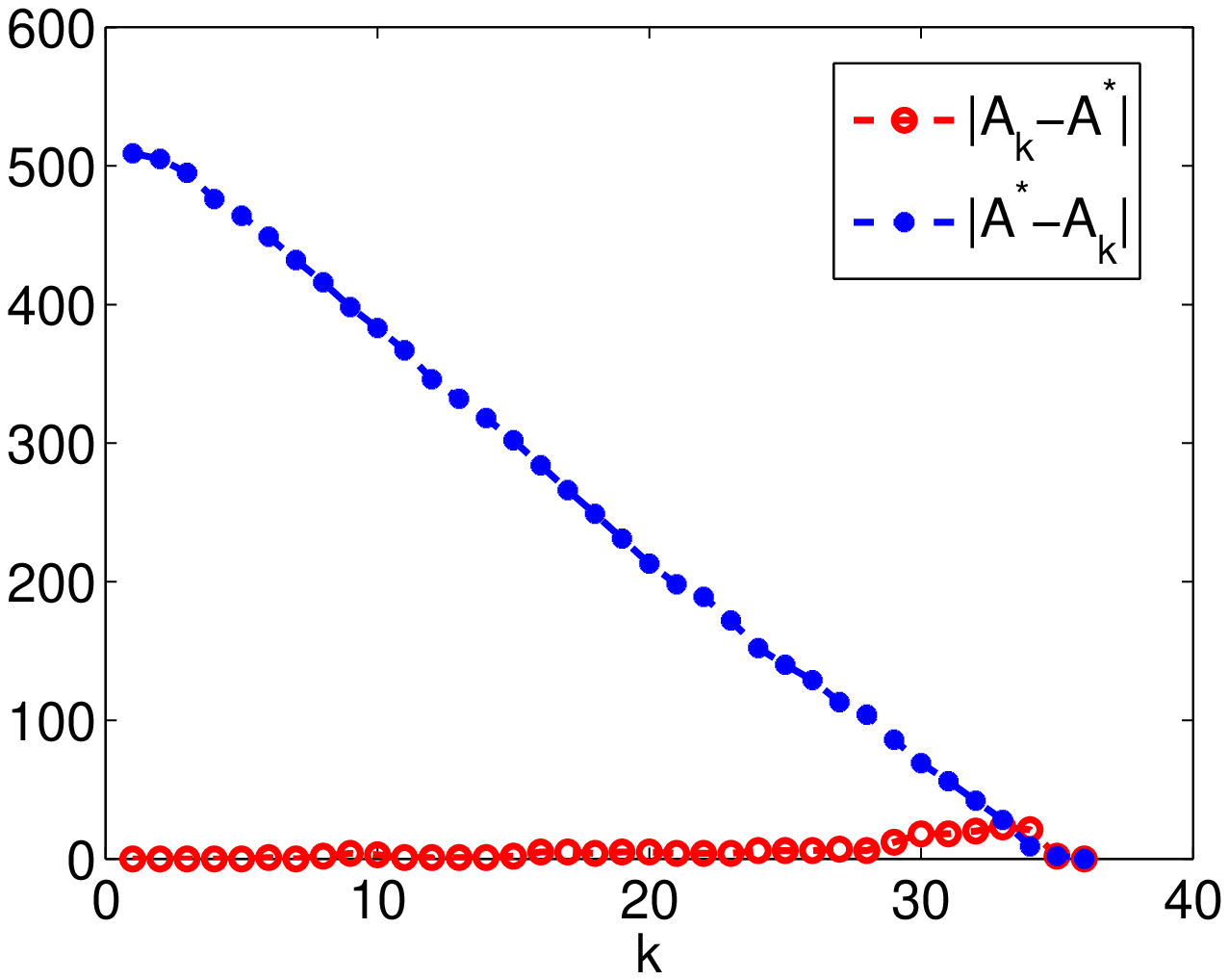}}\\
    (c) random Bernoulli, $T=2^8$ & (d) random Bernoulli, $T=2^9$\\
    \includegraphics[trim = 0cm 0cm 0cm 0cm, clip=true,width=6cm]{{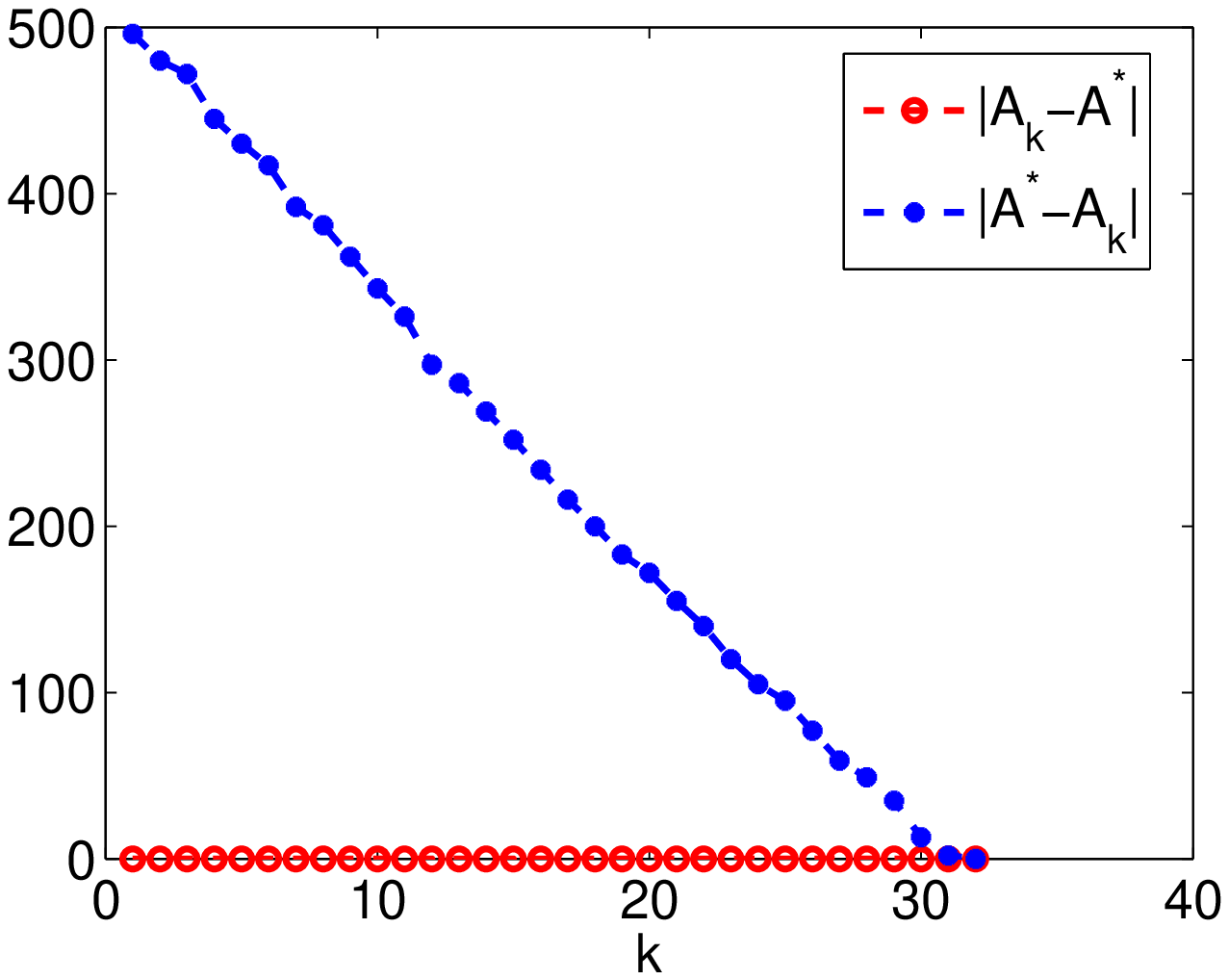}} &\includegraphics[trim = 0cm 0cm 0cm 0cm, clip=true,width=6cm]{{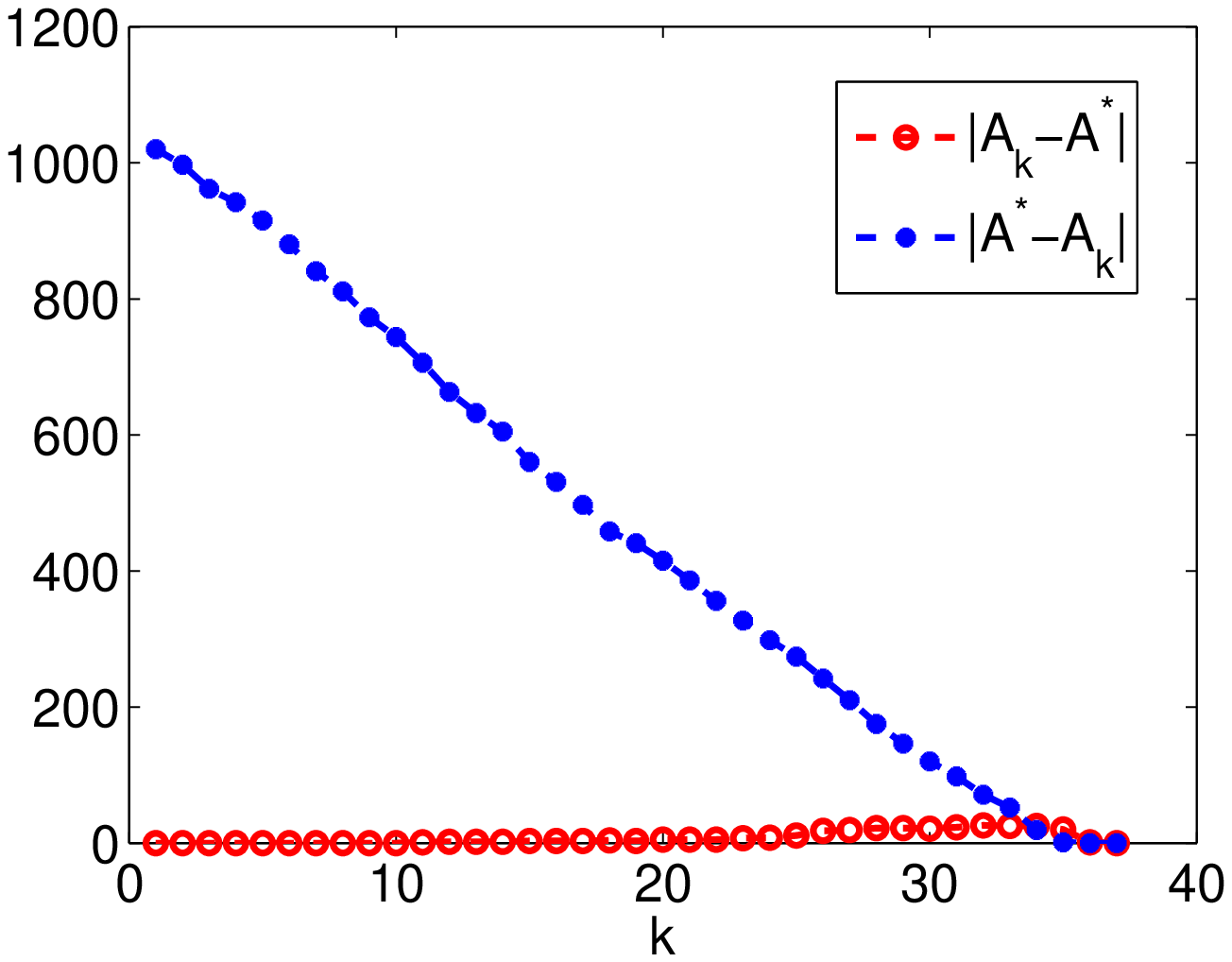}}\\
    (e) partial DCT, $T=2^9$ & (f) partial DCT, $T=2^{10}$
  \end{tabular}
  \caption{Numerical results for random Gaussian (top row, $R=100$, $n=500$, $p=1000$, $\sigma=\mbox{1e-3}$),
  random Bernoulli (middle row, $R = 1000$, $n = 2^{10}$, $p = 2^{12}$, $\sigma =\mbox{1e-3}$)
  and partial DCT (bottom row, $R = 1000$, $n = 2^{11}$, $p = 2^{13}$, $\sigma =\mbox{1e-3}$) sensing matrix.
  The parameters $N$ and $J_{max}$ are set to $N=50$ and $J_{max}=1$, respectively.
  }\label{fig:acinout}
%
%
\end{figure}

For each $\lambda_k$, with $x({\lambda_{k-1}})$  ($x({\lambda_0})=0$) as the initial guess, the PDASC
generally reaches convergence within a few iterations, cf. Fig. \ref{fig:convsup}, which is observed for
random Gaussian, random Bernoulli and partial DCT sensing matrices. This is attributed to the local
superlinear convergence of the PDAS algorithm. Hence, when coupled with the continuation strategy,
the PDASC procedure is very efficient.

\begin{figure}[htb!]
  \centering
  \begin{tabular}{cc}
   \includegraphics[trim = 0cm 0cm 1cm 0cm, clip=true,width=7cm]{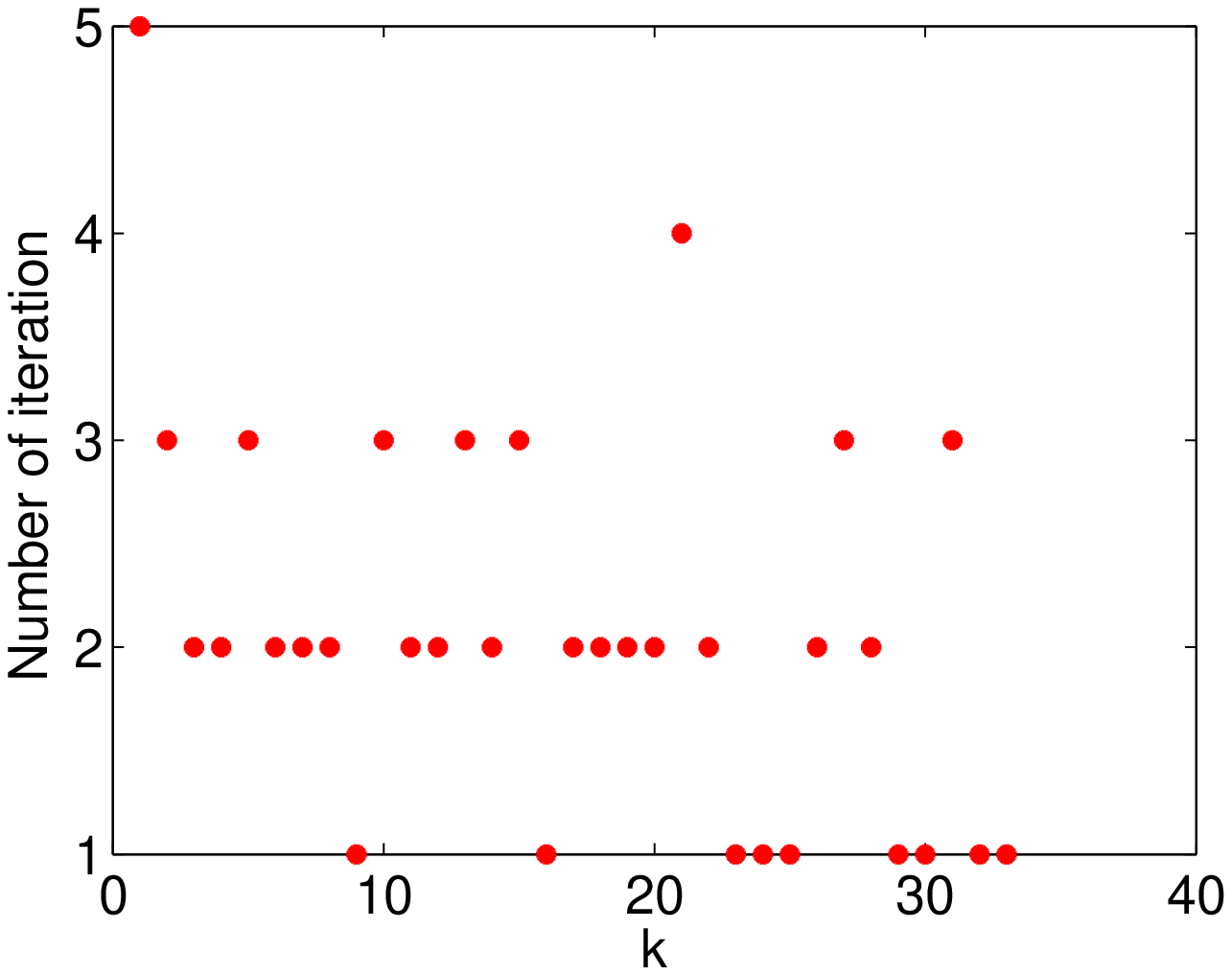} &
   \includegraphics[trim = 0cm 0cm 1cm 0cm, clip=true,width=7cm]{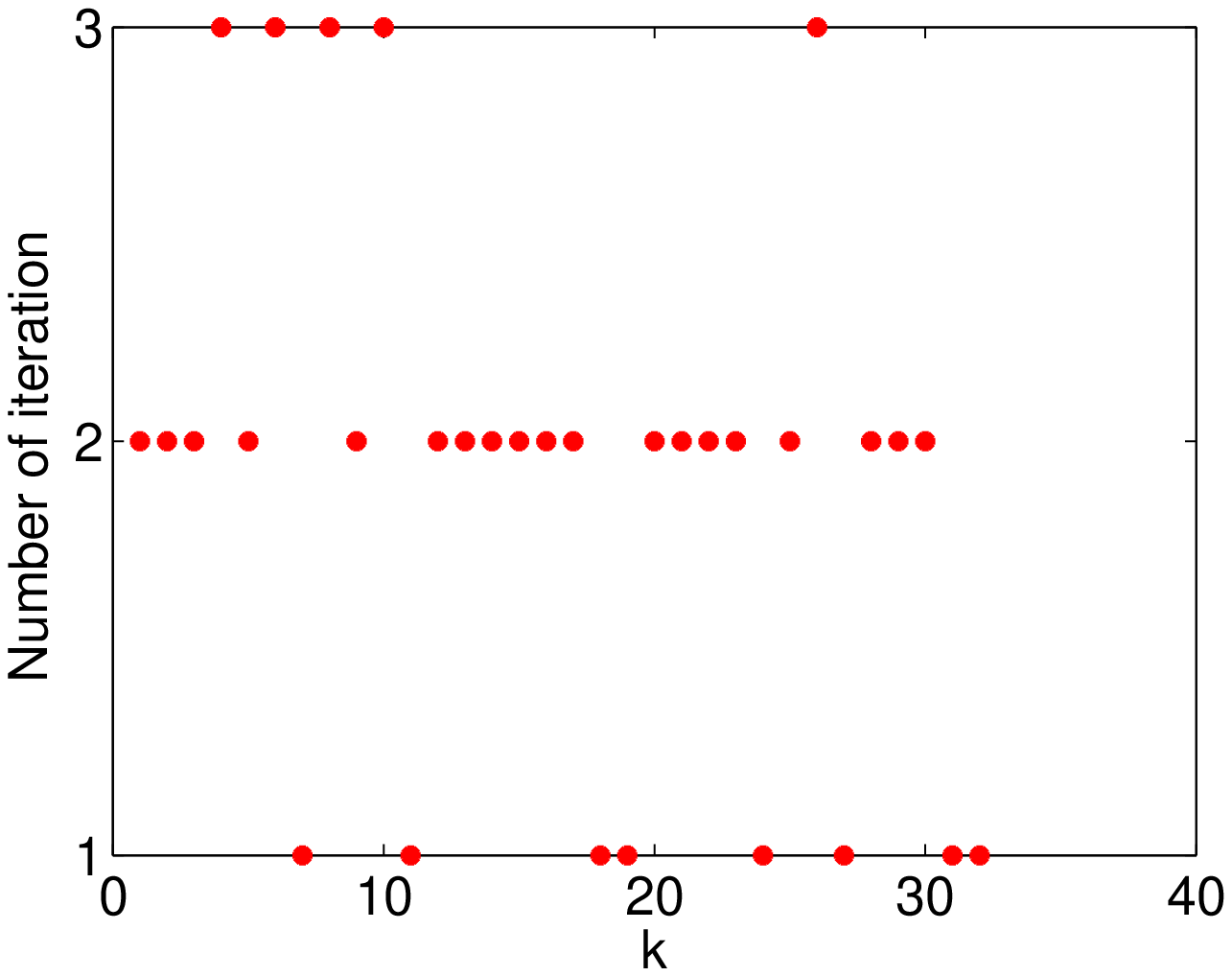}\\
   (a) Random Gaussian & (b) Random Bernoulli\\
   \includegraphics[trim = 0cm 0cm 1cm 0cm, clip=true,width=7cm]{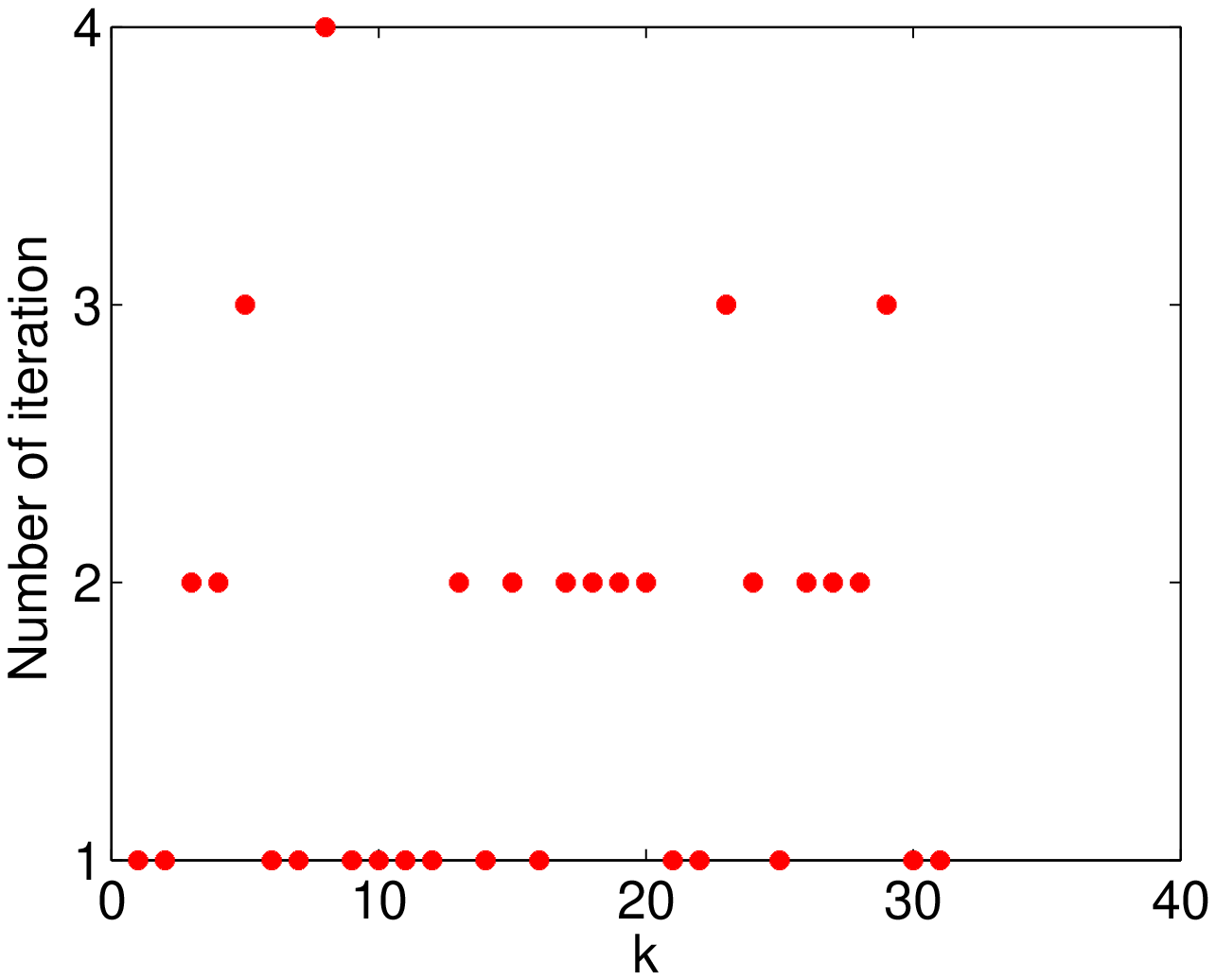}\\
   (c) Partial DCT \\
  \end{tabular}
   \caption{Number of iterations of PDASC at each $\lambda_k$ for random Gaussian (top left with $R=1000$, $n=500$, $p=1000$, $T= 200$, $\sigma=\mbox{1e-3}$),
  random Bernoulli (top right with $R = 1000$, $n = 2^{10}$, $p = 2^{12}$, $T = 2^{8}$, $\sigma =\mbox{1e-3}$)
  and partial DCT (bottom with  $R = 1000$, $n = 2^{11}$, $p = 2^{13}$, $T = 2^{8}$, $\sigma =\mbox{1e-3}$) sensing matrix.
  The parameters $N$ and $J_{max}$ are set to $N=50$ and $J_{max}=5$, respectively.}\label{fig:convsup}
\end{figure}

\subsection{Comparison with existing algorithms}
In this part, we compare Algorithm \ref{alg:pdasc} with six state-of-the-art
algorithms in the compressive sensing literature, including orthogonal matching pursuit (OMP) \cite{PatiRezaiifarKrishnaprasad:1993},
greedy gradient pursuit (GreedyGP) \cite{BlumensathDavies:2009a}, accelerated iterative
hard thresholding (AIHT) \cite{blumensath2012accelerated}, hard thresholding pursuit (HTP) \cite{Foucart:2011},
compressive sampling matching pursuit (CoSaMP) \cite{NeedellTropp:2009} and homotopy algorithm \cite{homtop3, homtop1}.

First, we consider the exact support recovery probability, i.e., the percentage of
the reconstructions whose support agrees with the true active set $A^*$.
To this end, we fix the sensing matrix $\Psi$ as a $500\times 1000$ random Gaussian matrix, $\sigma =
\mbox{1e-3}$, $(N,J_{max})=(100,5) $ or $(50,1)$, and all results are computed from $100$ independent realizations
of the problem setup. Since the different dynamical range may give different results, we take
$R=1$, $10$, 1e3, 1e5 as four exemplary values. The numerical results are summarized in
Fig. \ref{fig:sparslev0}. We observe that when the dynamical range $R$ is not very small,
the proposed PDASC algorithm with $(N,J_{max})=(100,5)$ has a better exact support recovery
probability, and that with the choice $(N,J_{max})$ is largely comparable with other algorithms.

\begin{figure}[ht!]
  \centering
  \begin{tabular}{cc}
    \includegraphics[trim = 1cm 0cm 1cm 0cm, clip=true,width=7cm]{{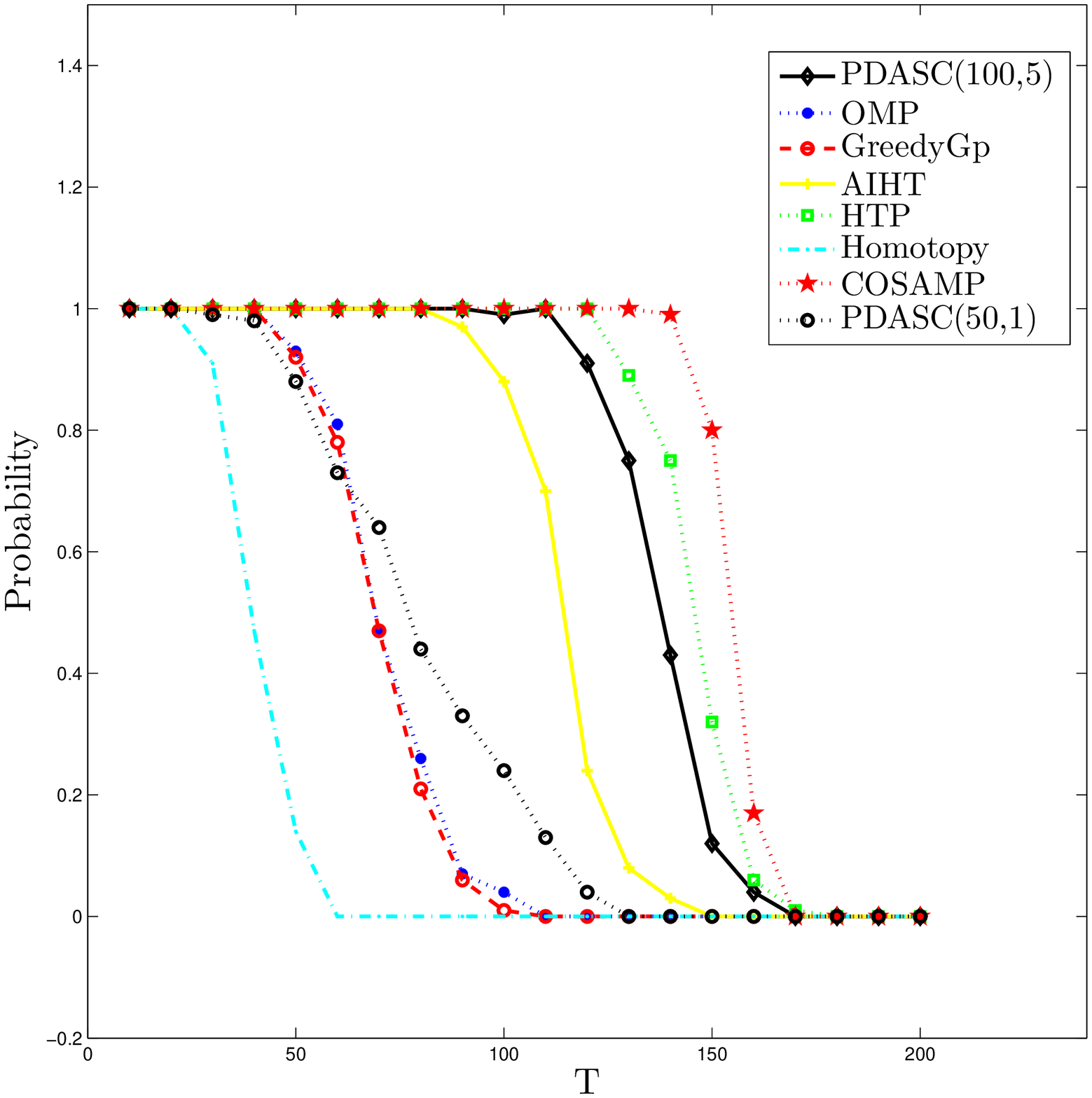}} &
    \includegraphics[trim = 1cm 0cm 1cm 0cm, clip=true,width=7cm]{{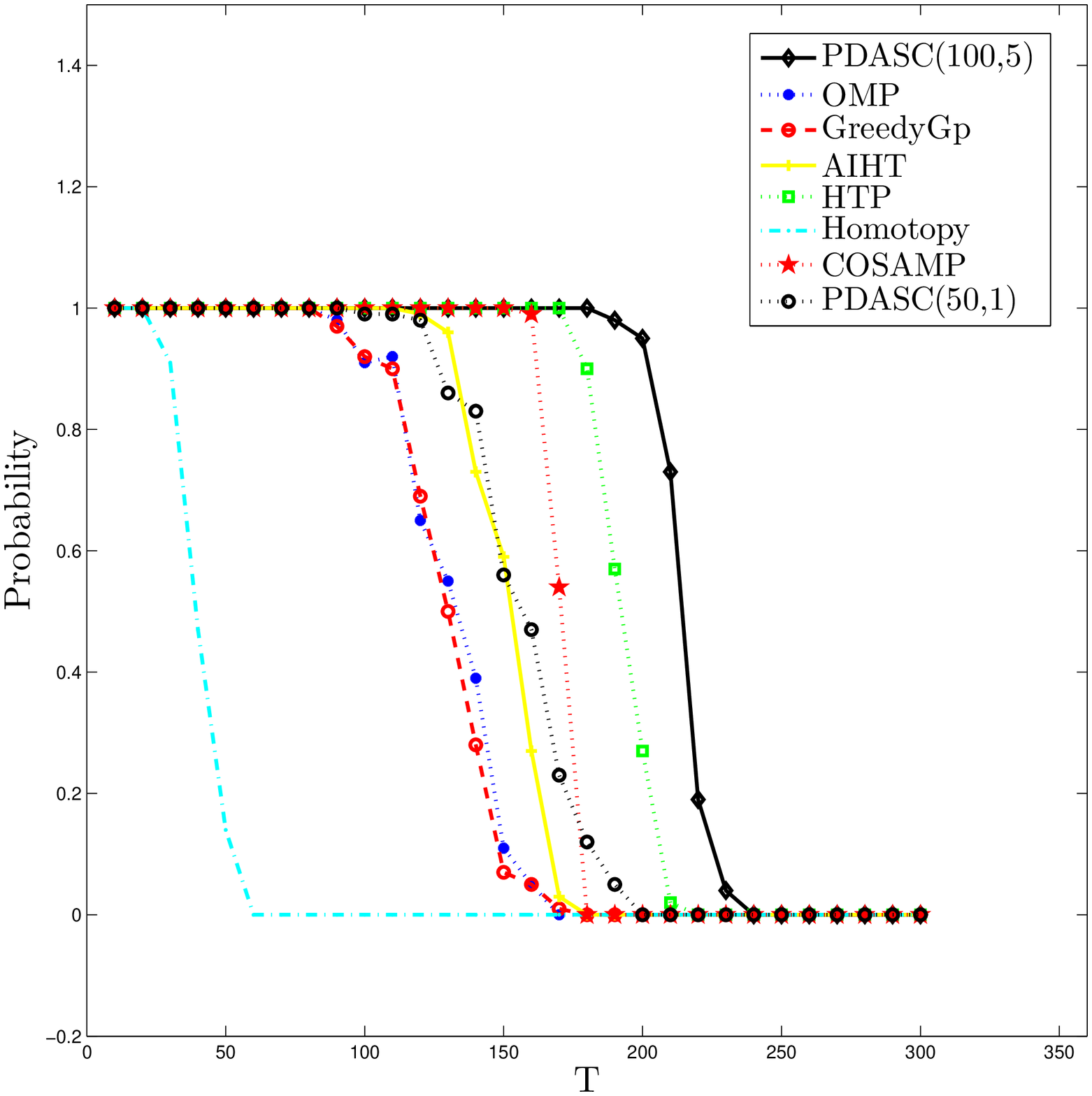}}\\
    (a) $R=1$ & (b) $R=10$\\
    \includegraphics[trim = 1cm 0cm 1cm 0cm, clip=true,width=7cm]{{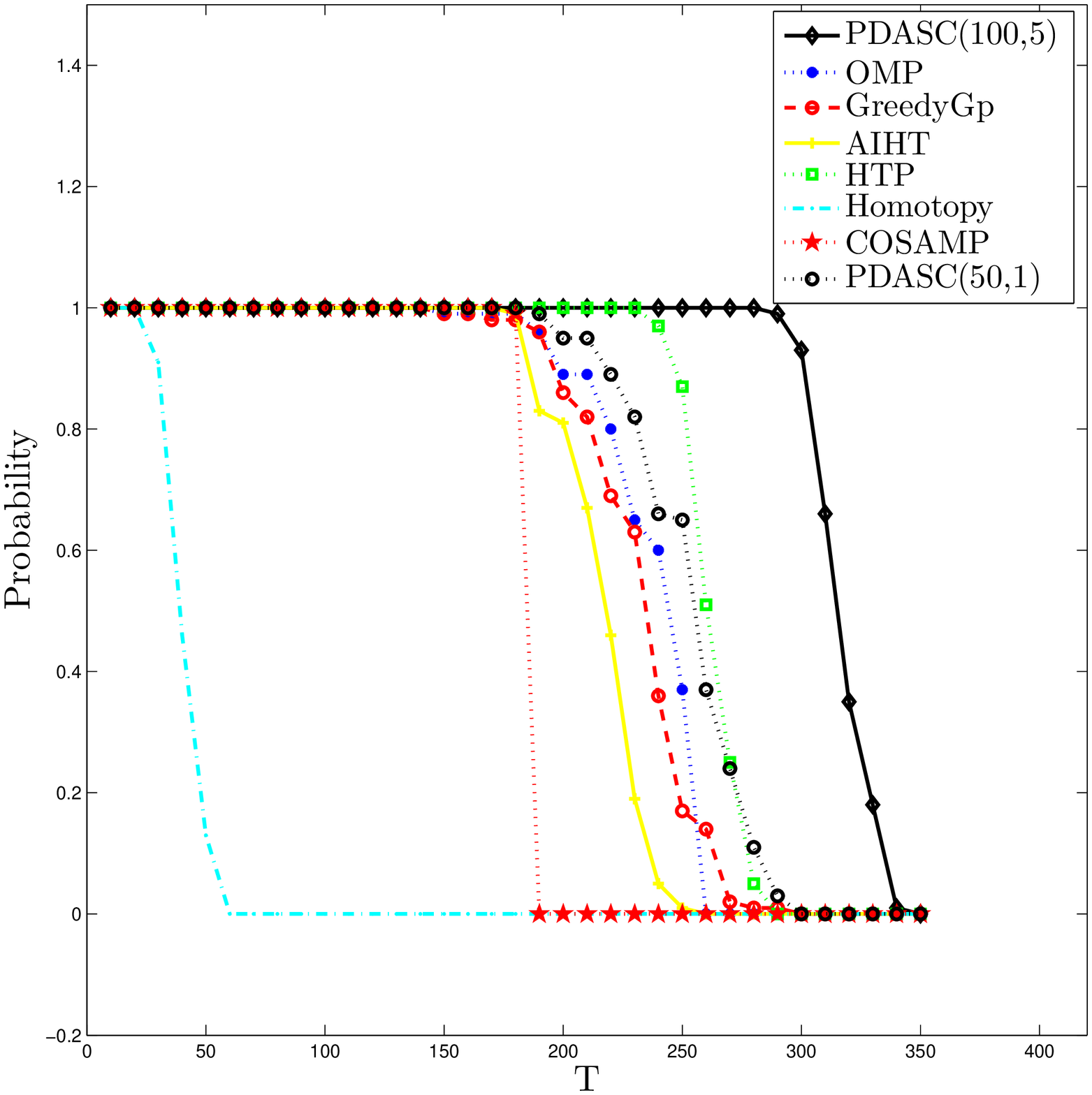}} &
    \includegraphics[trim = 1cm 0cm 1cm 0cm, clip=true,width=7cm]{{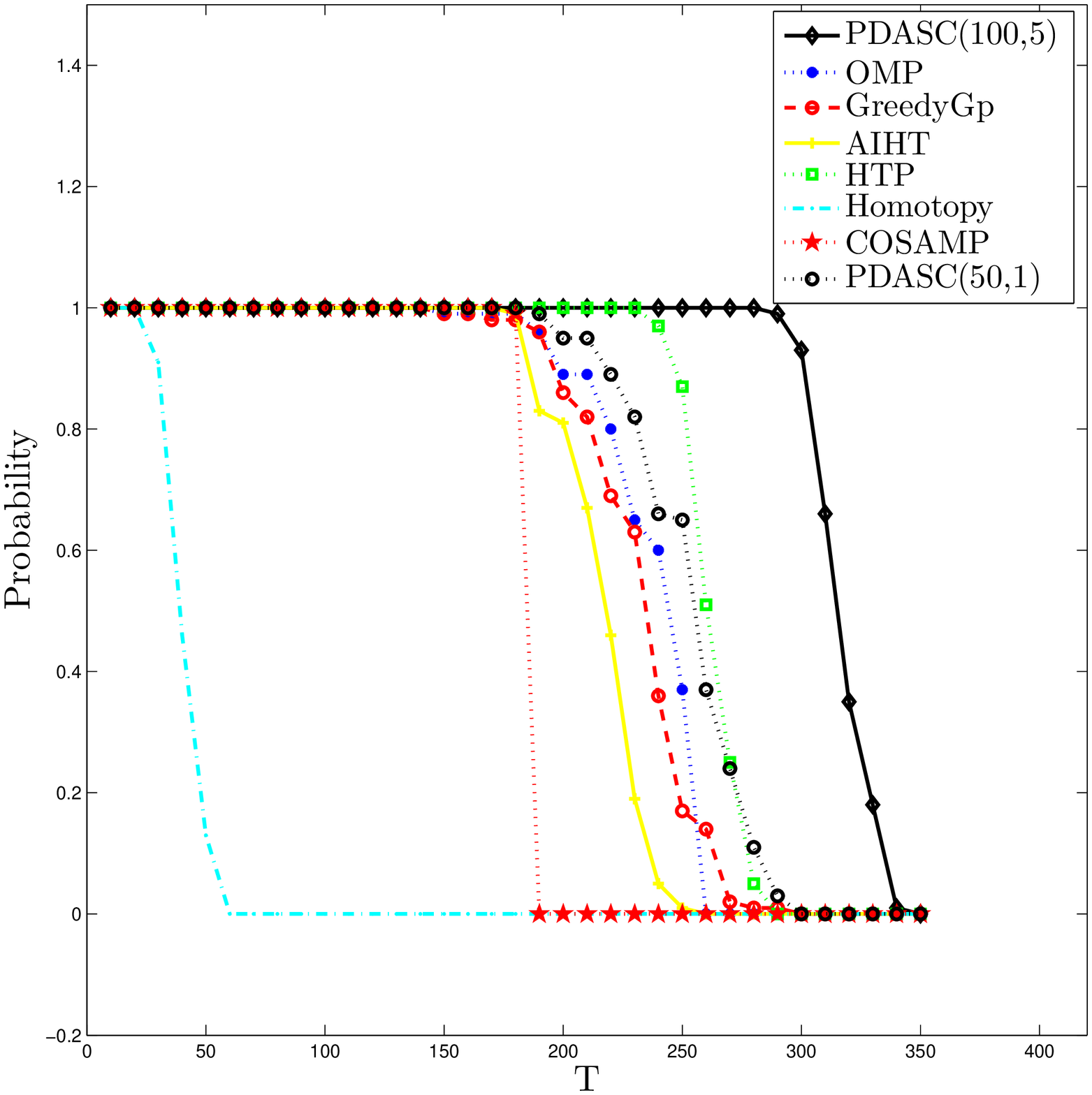}}\\
    (c) $R=10^3$ & (d) $R=10^5$
  \end{tabular}
  \caption{The exact support recovery probability for four different dynamical ranges: $R=1,\ 10,\ 10^3$, and $10^5$.}\label{fig:sparslev0}
\end{figure}

To further illustrate the accuracy and efficiency of the proposed PDASC algorithm, we compare it with
other greedy methods in terms of CPU time and reconstruction error. To this end, we fix $\sigma =
\mbox{1e-2}$, $(N,J_{max})=(100,5)$ or $(50,1)$. The numerical results for random Gaussian,
random Bernoulli and partial DCT sensing matrices with different parameter tuples $(R,n,p,T)$ are shown
in Tables \ref{tab:timeerrorg}-\ref{tab:timeerrorp}, respectively. The results in the tables are
computed from 10 independent realizations of the problem setup. It is observed that the PDASC
algorithm yields reconstructions that are comparable with that by other methods, but usually
with less computing time. Further, we observe that it scales
better with the problem size than other algorithms.

\begin{table}[htb!]
\centering
\caption{Numerical results (CPU time and errors) for medium-scale problems, with random Gaussian
sensing matrix $\Psi$, of size $p = 10000,\ 15000,\ 20000,\ 25000,\ 30000$, $n = \lfloor p/4\rfloor$,
$T= \lfloor n/3\rfloor$. The dynamical range $R$ is $R=1000$, and the noise variance $\sigma$ is
$\sigma=\mbox{1e-2}$.}\label{tab:timeerrorg}
\begin{tabular}{cccccp{0.3cm}p{0.3cm}c}
\hline\hline
\multicolumn{1}{c}{$p$} & \multicolumn{1}{c}{method} & \multicolumn{1}{c}{time(s)} & \multicolumn{1}{c}{Relative $\ell^2$ error} & \multicolumn{1}{c}{Absolute $\ell^{\infty}$ error} \\
 \hline
                  &PDASC(50,1)       &1.94     &4.24e-5         &3.91e-2            \\
                  &PDASC(100,5)    &6.41    &4.49e-5          &3.81e-2      \\
    $10000$       &OMP               &17.4     &4.24e-5         &3.91e-2            \\
                  &GreedyGP          &17.9     &7.50e-5         &1.35e-1          \\
                  &AIHT              &4.62     &4.24e-5         &3.91e-2            \\
                  &HTP               &1.69     &4.24e-5         &3.91e-2            \\
                  &CoSaMP            &10.4     &1.21e-1         &8.87e+1         \\
 \hline
                 &PDASC(50,1)       &4.17     &4.32e-5         &4.60e-2        \\
                  &PDASC(100,5)    &17.4     &4.60e-5         &4.55e-2        \\
  $15000$        &OMP               &56.8     &4.32e-5         &4.60e-2        \\

                 &GreedyGP          &61.4     &6.36e-5         &1.17e-1       \\
                 &AIHT              &9.88     &4.32e-5         &4.60e-2        \\
                 &HTP               &4.52     &4.32e-5         &4.60e-2        \\
                 &CoSaMP            &32.4     &8.93e-2         &7.24e+1       \\
 \hline
                 &PDASC(50,1)       &8.14    &4.50e-5         &4.68e-2             \\
                 &PDASC(100,5)    &34.8    &4.56e-5          &4.34e-2      \\
  $20000$        &OMP               &134     &4.50e-5         &4.68e-2             \\
                 &GreedyGP          &141     &6.12e-5         &1.11e-1           \\
                 &AIHT              &18.4    &4.50e-5         &4.67e-2             \\
                 &HTP               &9.51    &4.50e-5         &4.68e-2             \\
                 &CoSaMP            &79.7    &1.12e-1         &8.03e+1         \\
  \hline
                 &PDASC(50,1)       &13.7    &4.47e-5      &4.03e-2            \\
                 &PDASC(100,5)    &62.3    &4.55e-5          &4.61e-2      \\
  $25000$        &OMP               &260     &4.47e-5      &4.03e-2            \\
                 &GreedyGP          &275     &6.19e-5      &1.35e-1          \\
                 &AIHT              &26.8    &4.47e-5      &4.03e-2            \\
                 &HTP               &19.6    &4.47e-5      &4.03e-2            \\
                 &CoSaMP            &150     &1.07e-1      &7.96e+1        \\
  \hline
                 &PDASC(50,1)       &21.1    &4.52e-5        &4.90e-2             \\
                 &PDASC(100,5)      &95.9    &4.53e-5          &4.53e-2      \\
  $30000$        &OMP               &445     &4.52e-5        &4.90e-2             \\
                 &GreedyGP          &475     &5.57e-5        &1.02e-1           \\
                 &AIHT              &40.3    &4.52e-5        &4.90e-2             \\
                 &HTP               &29.8    &4.52e-5        &4.90e-2             \\
                 &CoSaMP            &251     &9.00e-2        &6.48e+1         \\
\hline\hline
\end{tabular}
\end{table}

\begin{table}[htb!]
\centering
\caption{Numerical results (CPU time and errors) for medium-scale problems, with random Bernoulli
sensing matrix $\Psi$, of size $p = 10000,\ 15000,\ 20000,\ 25000,\ 30000$, $n = \lfloor p/4\rfloor$,
$T= \lfloor n/4\rfloor$. The dynamical range $R$ is $R=10$, and the noise variance $\sigma$ is
$\sigma=\mbox{1e-2}$.}\label{tab:timeerrorb}
\begin{tabular}{cccccp{0.3cm}p{0.3cm}c}
\hline\hline
\multicolumn{1}{c}{$p$} & \multicolumn{1}{c}{method} & \multicolumn{1}{c}{time(s)} & \multicolumn{1}{c}{Relative $\ell^2$ error} & \multicolumn{1}{c}{Absolute $\ell^{\infty}$ error} \\
 \hline
                  &PDASC(50,1)     &0.62     &2.33e-3         &3.66e-2              \\
                  &PDASC(100,5)    &2.56    &2.50e-3          &3.82e-2      \\
    $10000$       &OMP             &9.92     &2.33e-3         &3.70e-2              \\
                  &GreedyGP        &11.8     &1.80e-2         &1.02e+0            \\
                  &AIHT            &3.50     &2.31e-3         &3.64e-2              \\
                  &HTP             &0.90     &2.31e-3         &3.64e-2              \\
                  &CoSaMP          &4.50     &4.68e-3         &6.71e-2           \\
 \hline
                 &PDASC(50,1)      &1.41     &2.39e-3         &3.62e-2         \\
                 &PDASC(100,5)    &6.44    &2.50e-3          &4.11e-2      \\
  $15000$        &OMP              &33.4     &2.40e-3         &3.63e-2         \\
                 &GreedyGP         &40.1     &1.86e-2         &1.03e+0        \\
                 &AIHT             &7.19     &2.39e-3         &3.62e-2         \\
                 &HTP              &2.27     &2.39e-3         &3.62e-2         \\
                 &CoSaMP           &13.6     &5.06e-3         &8.07e-2        \\
 \hline
                 &PDASC(50,1)      &2.67     &2.52e-3         &4.21e-2             \\
                 &PDASC(100,5)    & 12.7    &2.50e-3         &3.97e-2      \\
  $20000$        &OMP              &78.5     &2.53e-3         &4.18e-2             \\
                 &GreedyGP         &94.9     &1.55e-2         &1.02e+0           \\
                 &AIHT             &12.8     &2.52e-3         &4.20e-2             \\
                 &HTP              &4.25     &2.52e-3         &4.20e-2             \\
                 &CoSaMP           &30.1     &4.99e-3         &7.69e-2         \\
  \hline
                 &PDASC(50,1)      &4.51     &2.49e-3         &4.16e-2         \\
                 &PDASC(100,5)    &21.7    &2.50e-3          &3.99e-2      \\
  $25000$        &OMP              &152.     &2.50e-3         &4.21e-2         \\
                 &GreedyGP         &185.     &2.14e-2         &1.07e+0       \\
                 &AIHT             &23.9     &2.49e-3         &4.17e-2         \\
                 &HTP              &7.78     &2.49e-3         &4.17e-2         \\
                 &CoSaMP           &58.4     &5.04e-3         &8.13e-2     \\
  \hline
                 &PDASC(50,1)      &6.90     &2.45e-3        &8.67e-2             \\
                 &PDASC(100,5)    &34.7    &2.50e-3          &4.18e-2      \\
  $30000$        &OMP              &264.     &2.46e-3        &1.01e-2             \\
                 &GreedyGP         &324.     &1.72e-2        &1.01e+0           \\
                 &AIHT             &24.7     &2.45e-3        &8.47e-2             \\
                 &HTP              &10.7     &2.45e-3        &8.46e-2             \\
                 &CoSaMP           &93.3     &4.97e-3        &3.63e+1         \\
\hline\hline
\end{tabular}
\end{table}

\begin{table}[htb!]
\centering
\caption{Numerical results (CPU time and errors) for large-scale problems, with partial DCT
sensing matrix $\Psi$, of size $p = 2^{13}$, $2^{14}$, $2^{15}$, $2^{16}$, $2^{17}$, $n = \lfloor p/4\rfloor$,
$T= \lfloor n/3\rfloor$. The dynamical range $R$ is $R=100$, and the noise variance $\sigma$ is
$\sigma=\mbox{1e-2}$.}\label{tab:timeerrorp}
\begin{tabular}{cccccp{0.3cm}p{0.3cm}c}
\hline\hline
\multicolumn{1}{c}{$p$} & \multicolumn{1}{c}{method} & \multicolumn{1}{c}{time(s)} & \multicolumn{1}{c}{Relative $\ell^2$ error} & \multicolumn{1}{c}{Absolute $\ell^{\infty}$ error} \\
 \hline
                  &PDASC(50,1)    &0.31    &7.11e-4         &7.78e-2      \\
                  &PDASC(100,5)   &1.51    &7.08e-4         &7.93e-2      \\
    $p = 2^{13}$  &OMP            &2.26    &7.06e-4         &7.75e-2        \\
                  &GreedyGP       &0.74    &1.07e-3         &1.79e-1        \\
                  &AIHT           &0.35    &7.06e-4         &7.75e-2       \\
                  &HTP            &0.57    &7.06e-4         &7.75e-2          \\
                  &CoSaMP         &0.55    &3.74e-1         &3.10e+1       \\
 \hline
                 &PDASC(50,1)     &0.48    &7.05e-4         &7.43e-2          \\
                 &PDASC(100,5)    &3.30    &6.95e-4         &8.49e-2      \\
  $2^{14}$        &OMP            &11.3    &7.00e-4         &7.49e-2          \\
                 &GreedyGP        &2.52    &1.21e-3         &5.95e-1         \\
                 &AIHT            &0.48    &7.01e-4         &7.53e-2          \\
                 &HTP             &0.95    &7.01e-4         &7.53e-2          \\
                 &CoSaMP          &0.87    &3.52e-1         &2.95e+1         \\
 \hline
                 &PDASC(50,1)     &0.85    &8.68e-4         &5.52e-1              \\
                 &PDASC(100,5)    &6.59    &7.42e-4         &1.74e-1      \\
  $ 2^{15}$        &OMP           &65.6    &9.94e-4         &1.00e-2              \\
                 &GreedyGP        &10.0    &8.08e-4         &1.95e-1            \\
                 &AIHT            &0.90    &7.15e-4         &8.27e-2              \\
                 &HTP             &1.73    &7.15e-4         &8.27e-2              \\
                 &CoSaMP          &1.55    &3.90e-1         &4.11e+1          \\
  \hline
                 &PDASC(50,1)     &1.60    &7.33e-4         &8.67e-2           \\
                 &PDASC(100,5)    &13.2    &7.07e-4         &8.81e-2      \\
  $2^{16}$        &OMP            &410    &1.13e-3          &1.01e+0            \\
                 &GreedyGP        &35.8    &1.00e-3         &1.01e+0         \\
                 &AIHT            &2.00    &7.28e-4         &8.47e-2           \\
                 &HTP             &3.59    &7.28e-4         &8.46e-2           \\
                 &CoSaMP          &2.75    &3.74e-1         &3.63e+1       \\
  \hline
                 &PDASC(50,1)     &4.37    &7.17e-4         &1.00e-1               \\
                 &PDASC(100,5)    &37.1    &7.13e-4         &9.86e-2      \\
  $ 2^{17}$        &OMP           &3.11e+3 &1.11e-3         &1.02e+0            \\
                 &GreedyGP        &199     &7.96e-4         &5.98e-1              \\
                 &AIHT            &4.90    &7.14e-4         &1.00e-1               \\
                 &HTP             &9.03    &7.14e-4         &1.00e-1               \\
                 &CoSaMP          &7.70    &3.90e-1         &3.90e+1           \\
\hline\hline
\end{tabular}
\end{table}

Lastly, we consider one-dimensional signals and two-dimensional images. In this case the explicit
form of the sensing matrix $\Psi$ may be not available, hence the least-squares step (for updating the primal variable)
at line 7 of Algorithm \ref{alg:pdasc} can only be solved by an iterative method. We employ the conjugate
gradient (CG) method to solve the least-squares problem inexactly. The initial guess for the CG method
for the $\lambda_k$-problem is the solution $x(\lambda_{k-1})$, and the stopping criterion for the
CG method is as follows: either the number of CG iterations is greater than 2 or the residual is
below a given tolerance $\mbox{1e-5}\epsilon$.

For the one-dimensional signal, the sampling matrix $\Psi$ is of size $665\times 1024$, and it consists
of applying a partial FFT and an inverse wavelet transform, and the signal under wavelet transformation
has $247$ nonzero entries and $\sigma=\mbox{1e-4}$, $N=50$, $J_{\max}=1$. The results are shown in
Fig. \ref{fig:1d} and Table \ref{tab:1d}. The reconstructions by all the methods, except the AIHT and CoSaMP,
are visually very appealing and in excellent agreement with the exact solution. The reconstructions by the AIHT
and CoSaMP suffer from pronounced oscillations. This is further confirmed by the PSNR values which is defined as
\begin{equation*}
 \textit{PSNR}=10\cdot \log\frac{V^2}{MSE}
\end{equation*}
where $V$ is the maximum absolute value of the reconstruction and the true solution, and $MSE$ is the mean
squared error of the reconstruction, cf. Table \ref{tab:1d}.

For the two-dimensional MRI image, the sampling matrix $\Psi$
amounts to a partial FFT and an inverse wavelet transform of size $1657\times 4096$. The image under
wavelet transformation has $792$ nonzero entries and $\sigma=\mbox{1e-4}$, $N=50$, and $J_{\max}=1$.
The numerical results are shown in Fig. \ref{fig:2d} and Table \ref{tab:2d}. The observation for the
one-dimensional signal remains largely valid: except the CoSaMP, all other methods can yield almost
identical reconstructions within similar computational efforts. Therefore, the proposed PDASC algorithm
is competitive with state-of-the-art algorithms.

\begin{figure}[ht!]
  \centering
  \begin{tabular}{cc}
    \includegraphics[trim = 2cm 2cm 1cm 1cm, clip=true,width=5cm]{{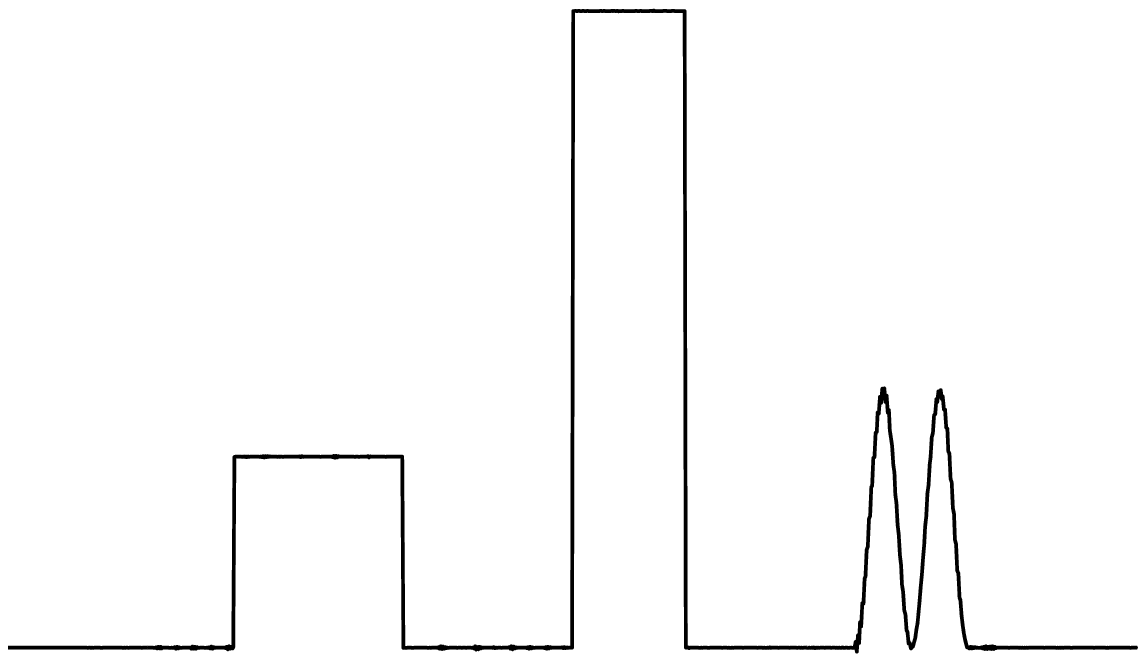}} &
    \includegraphics[trim = 2cm 2cm 1cm 1cm, clip=true,width=5cm]{{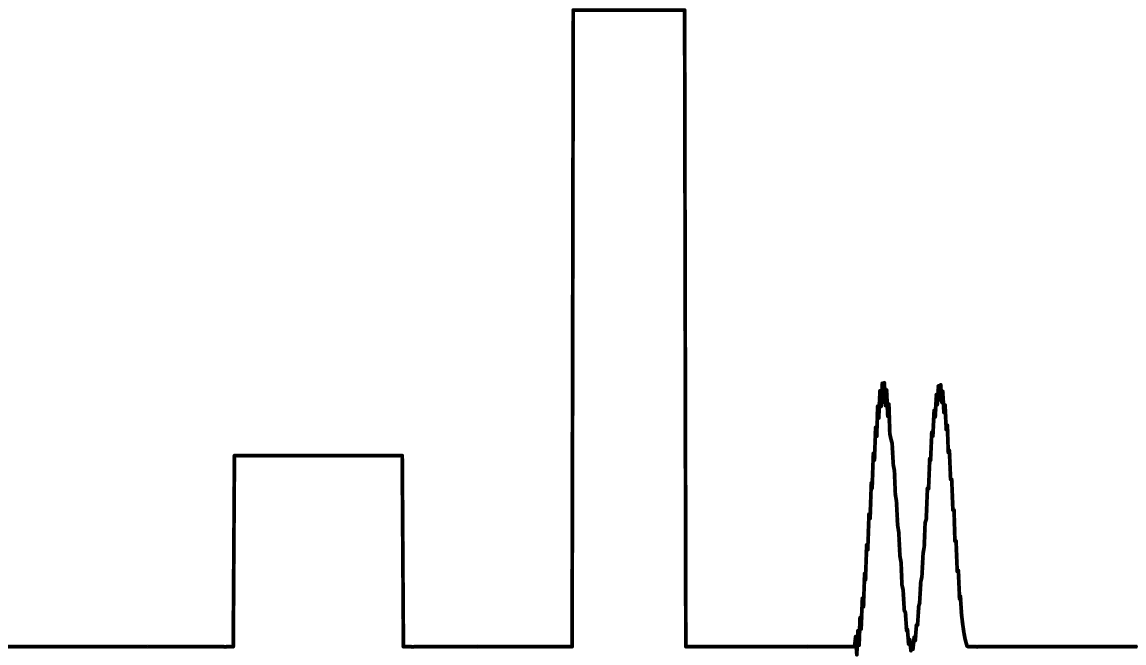}}\\
    (a) PDASC & (b) OMP\\
    \includegraphics[trim = 2cm 2cm 1cm 1cm, clip=true,width=5cm]{{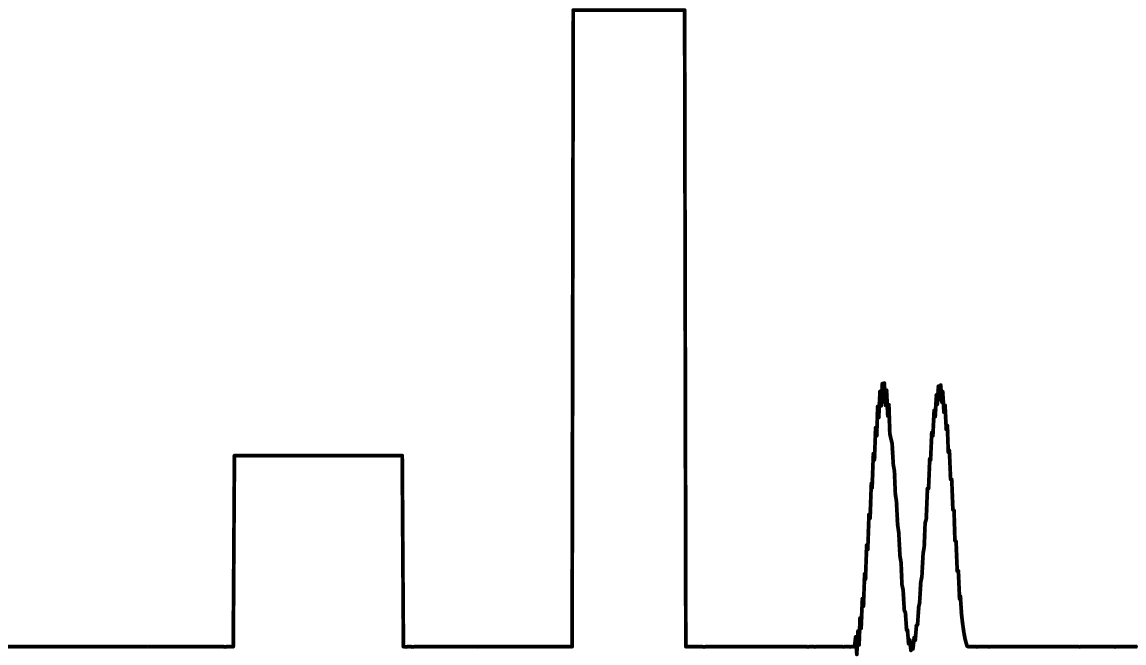}} &
    \includegraphics[trim = 2cm 2cm 1cm 1cm, clip=true,width=5cm]{{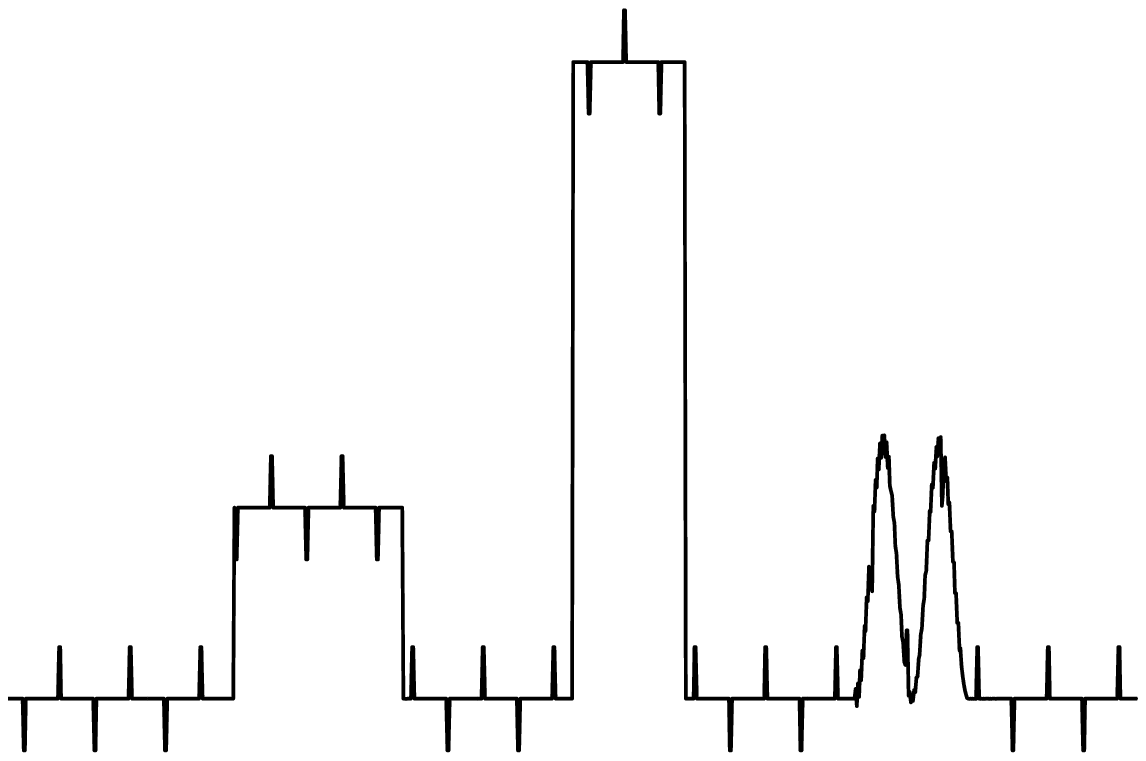}}\\
    (c) GreedyGP & (d) AIHT\\
    \includegraphics[trim = 2cm 2cm 1cm 1cm, clip=true,width=5cm]{{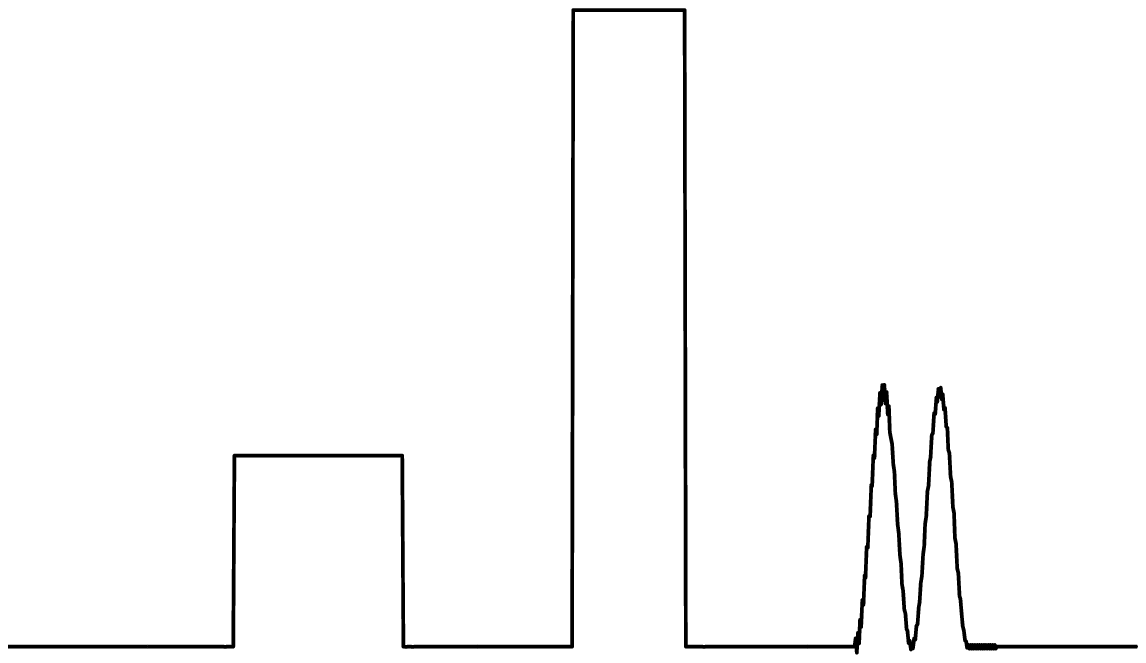}} &
    \includegraphics[trim = 2cm 2cm 1cm 1cm, clip=true,width=5cm]{{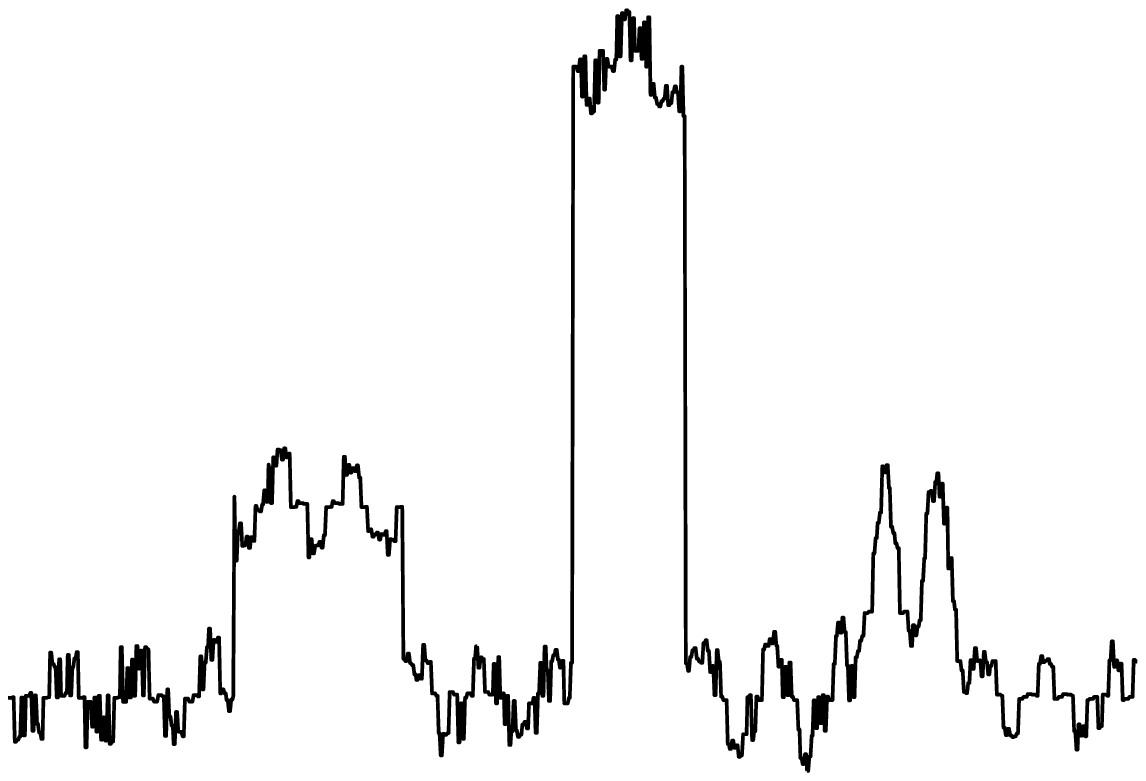}}\\
    (e) HTP & (f) CoSaMP
  \end{tabular}
  \caption{Reconstruction results of 1 dimension signal.}\label{fig:1d}
\end{figure}



\begin{figure}[ht!]
  \centering
  \begin{tabular}{cc}
    \includegraphics[trim = 2cm 1cm 1cm 1cm, clip=true,width=5cm]{{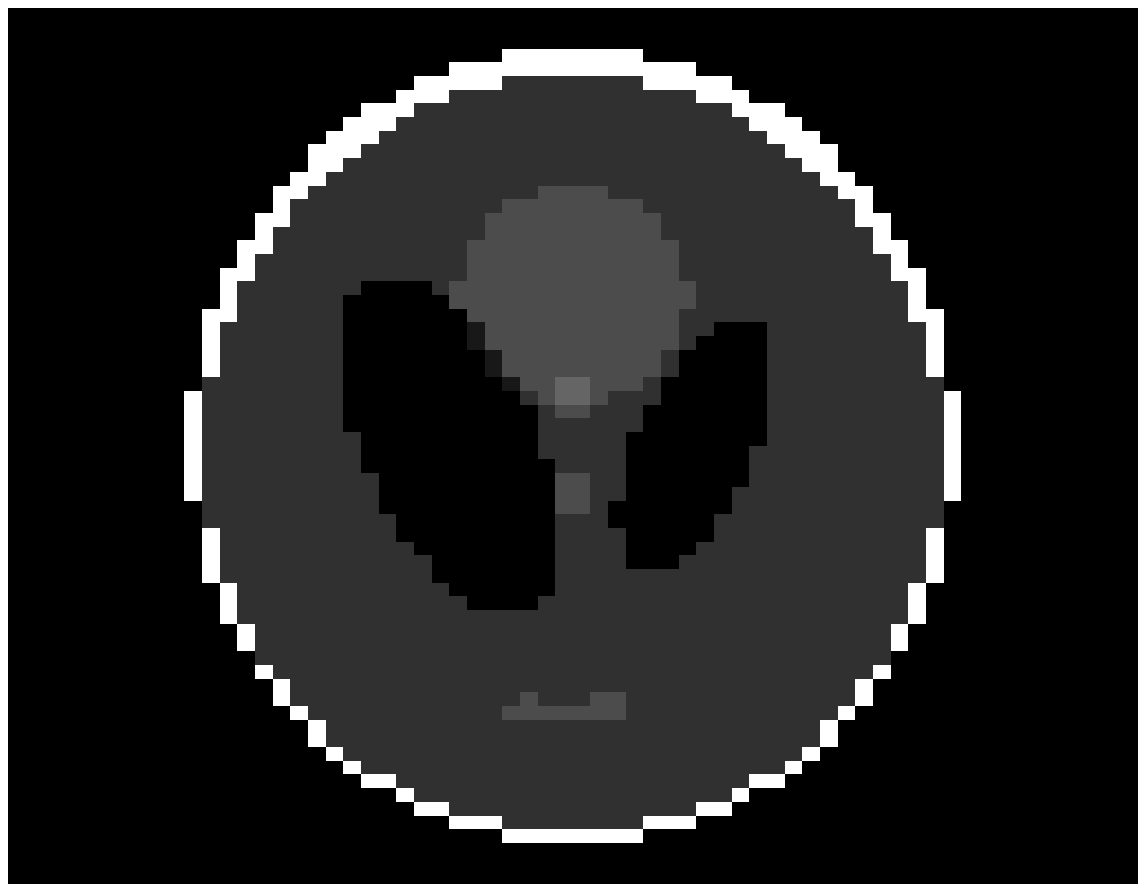}} &
    \includegraphics[trim = 2cm 1cm 1cm 1cm, clip=true,width=5cm]{{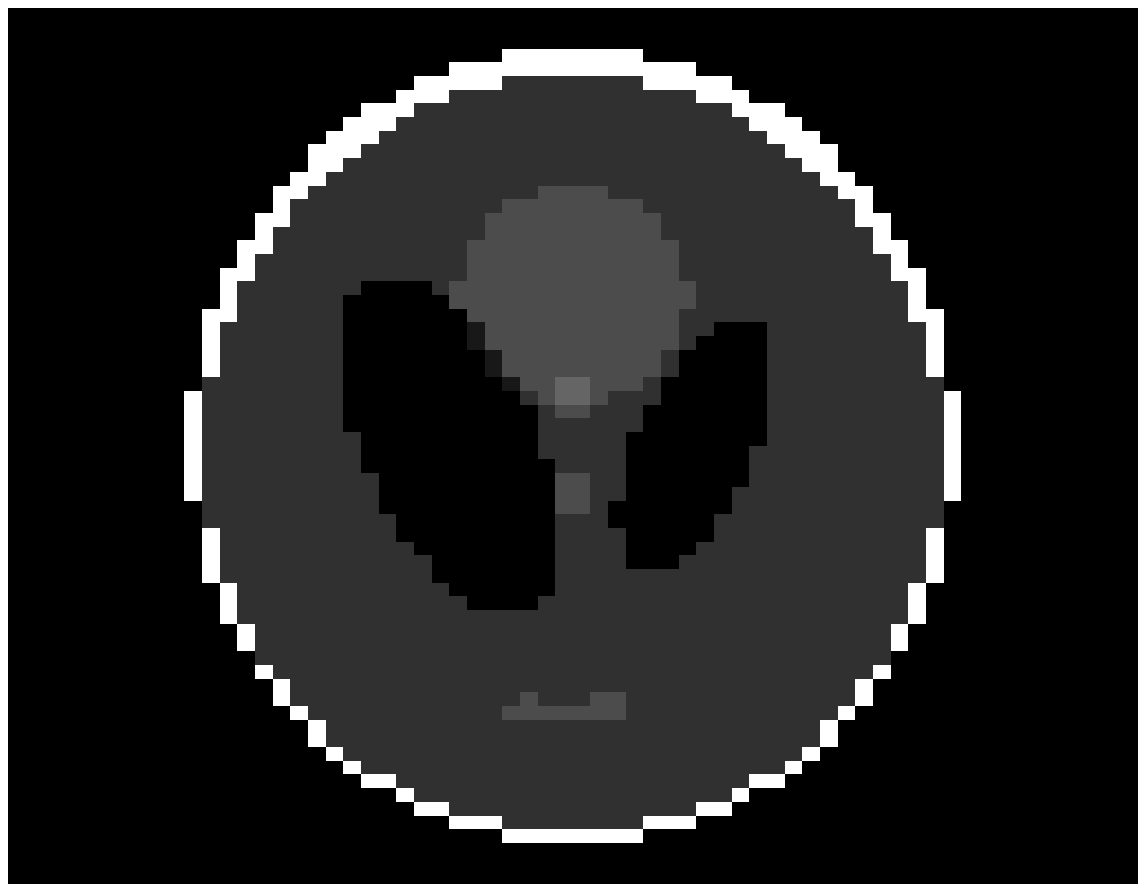}}\\
    (a) PDASC & (b) OMP\\
    \includegraphics[trim = 2cm 1cm 1cm 1cm, clip=true,width=5cm]{{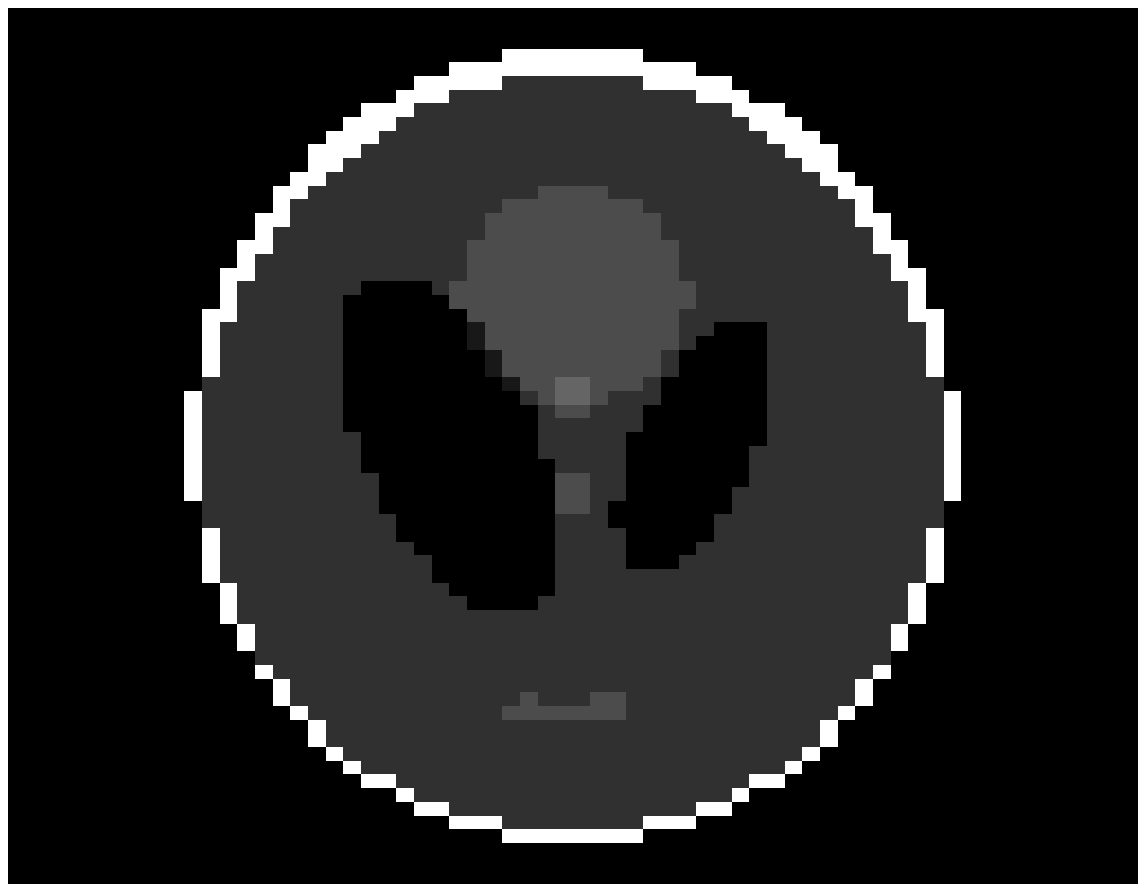}} &
    \includegraphics[trim = 2cm 1cm 1cm 1cm, clip=true,width=5cm]{{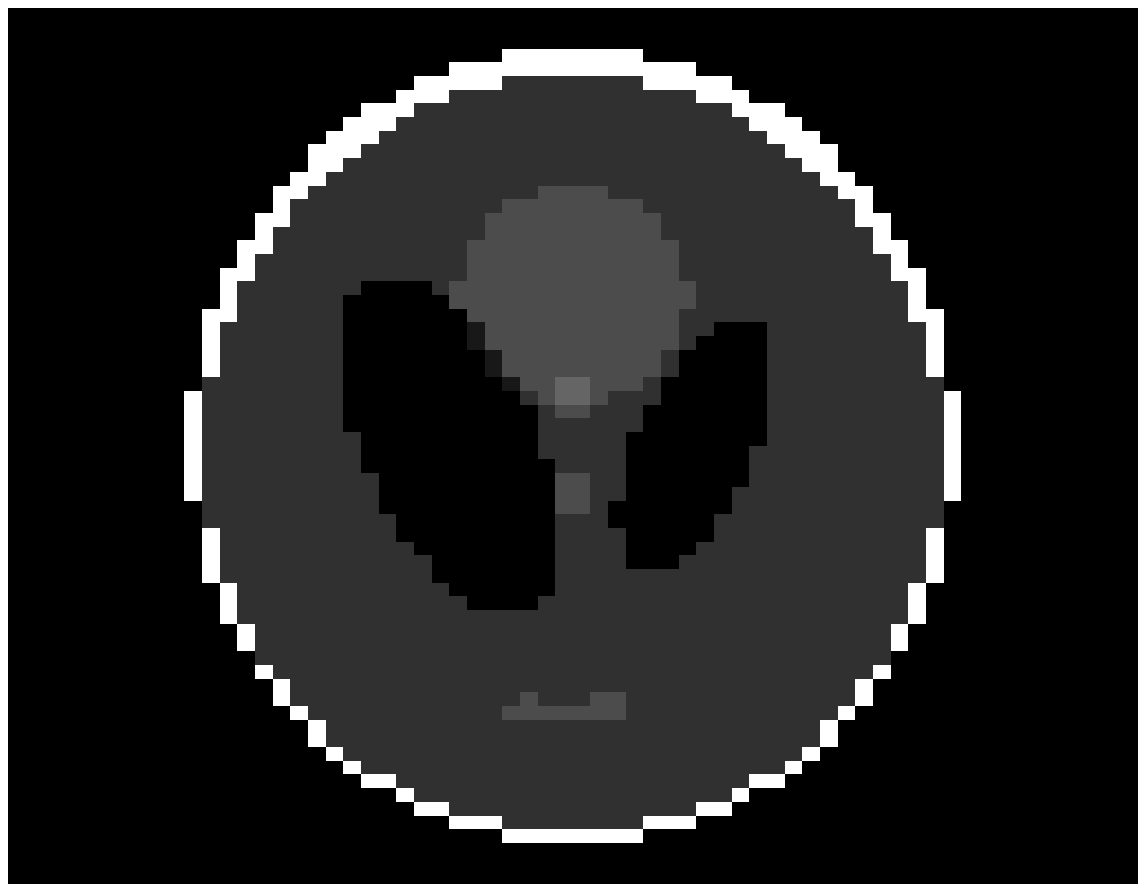}}\\
    (c) GreedyGP & (d) AIHT\\
    \includegraphics[trim = 2cm 1cm 1cm 1cm, clip=true,width=5cm]{{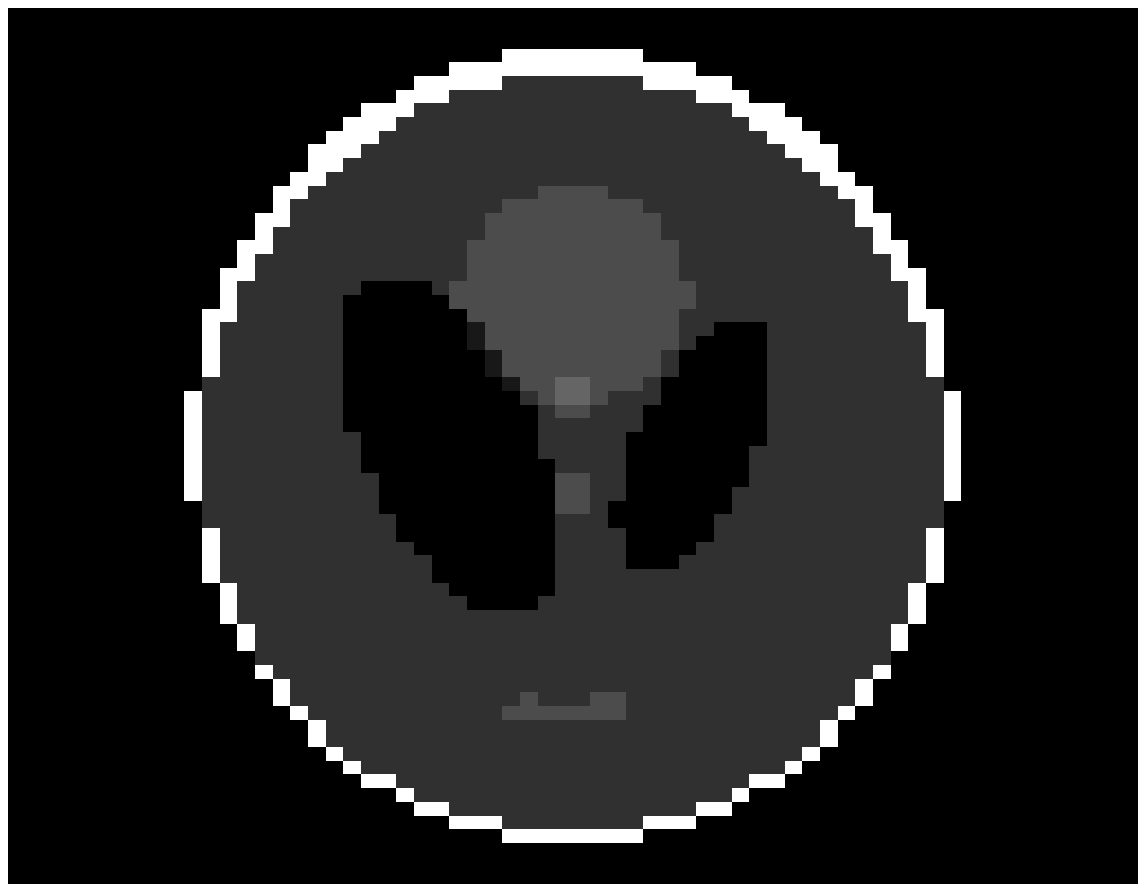}} &
    \includegraphics[trim = 2cm 1cm 1cm 1cm, clip=true,width=5cm]{{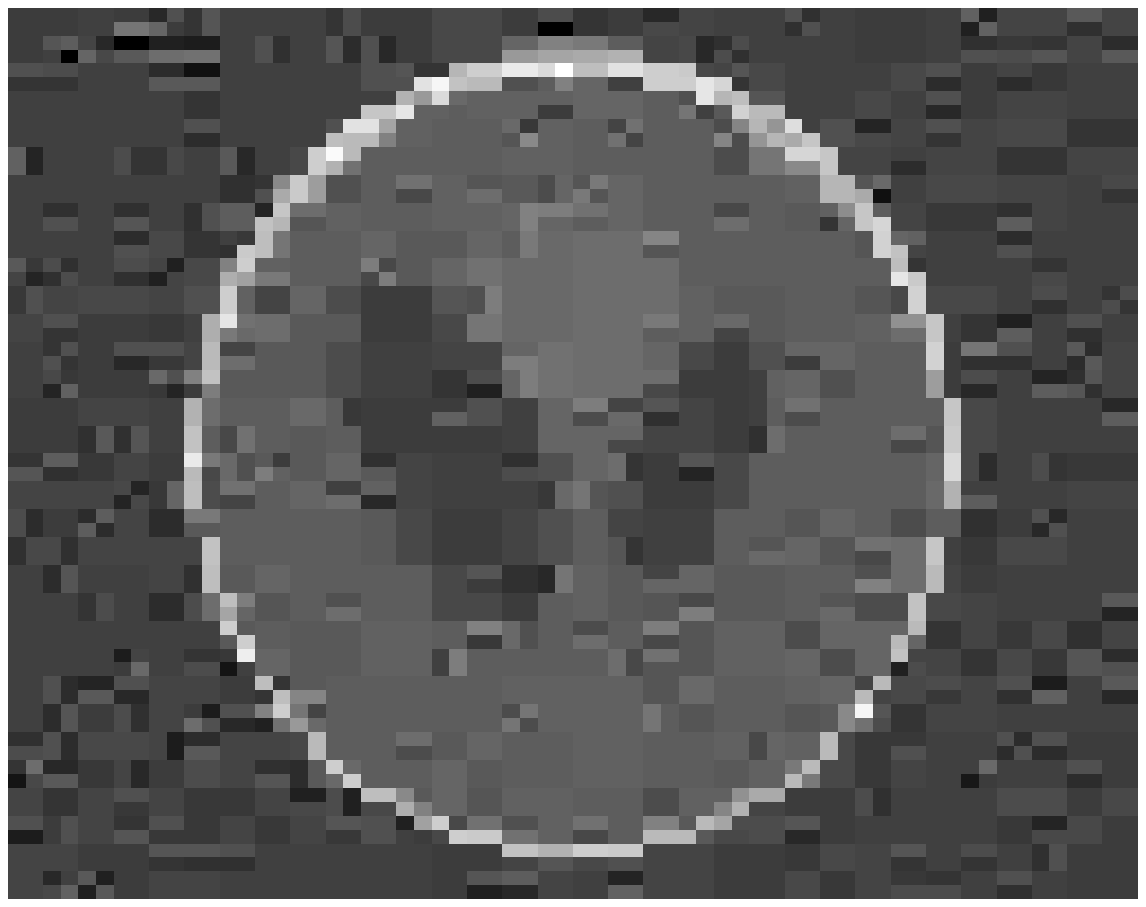}}\\
    (e) HTP & (f) CoSaMP
  \end{tabular}
  \caption{Reconstruction results of two-dimensional phantom images.}\label{fig:2d}
\end{figure}

\begin{table}[h]
\centering
 \caption{One-dimensional signal: $n=665$, $p=1024$, $T=247$, $\sigma=\mbox{1e-4}$.}\label{tab:1d}
 \begin{tabular}{ccccc}
 \hline
    method   &CPU time  &PSNR      \\
 \hline
  PDASC        & 0.49  &53     \\
  OMP          & 1.45  &49     \\
  GreedyGp     & 0.76  &49     \\
  AIHT         & 0.58  &34    \\
  HTP          &0.40  &51     \\
  CoSaMP       & 0.91 &26     \\
  \hline
  \end{tabular}
\end{table}

\begin{table}[h]
\centering
 \caption{Two-dimensional image, $n=1657$, $p=4096$, $T=792$,  $\sigma=\mbox{1e-4}$.}\label{tab:2d}
 \begin{tabular}{ccccc}
 \hline
    method   &CPU time  &PSNR      \\
 \hline
  PDASC          & 0.71  &81     \\
  OMP             & 6.82  &83     \\
  GreedyGP       & 2.93  &74       \\
  AIHT           & 0.56  &74    \\
  HTP            & 0.57 &82     \\
  CoSaMP          & 1.46 &20     \\
  \hline
  \end{tabular}
\end{table}

%
%
%
%

\section{Conclusion}
We have developed an efficient and accurate primal-dual active set with continuation algorithm
for the $\ell^0$ penalized least-squares problem arising in compressive sensing. It combines
the fast local convergence of the active set technique and the globalizing property of the
continuation technique. The global finite step convergence of the algorithm was established
under the mutual incoherence property or restricted isometry property on the sensing matrix.
Our extensive numerical results indicate that the proposed algorithm is competitive in comparison
with state-of-the-art algorithms in terms of efficiency, accuracy and exact recovery probability,
without a knowledge of the exact sparsity level.


\section*{Acknowledgement}
The research of B. Jin is supported by NSF Grant DMS-1319052, and that of X. Lu
is partially supported by National Science Foundation of  China No. 11101316 and No. 91230108.

\bibliographystyle{abbrv}
\bibliography{l0conv}
\end{document}